\documentclass{lms}
\usepackage[english]{babel}
\usepackage{amsmath}
\usepackage{amssymb}
\usepackage{psfrag,graphicx}
\usepackage{enumerate}
\usepackage{mathrsfs,hyperref}
\usepackage{epsfig}
\usepackage{pst-grad} 
\usepackage{pst-plot} 

\usepackage{xcolor}
\usepackage {tikz}
\usetikzlibrary {positioning}

\newtheorem{theorem}{Theorem}[section] 
\newtheorem{lemma}[theorem]{Lemma}     

\newtheorem{proposition}[theorem]{Proposition}

\newnumbered{assertion}{Assertion}    
\newnumbered{conjecture}{Conjecture}  
\newnumbered{definition}{Definition}
\newnumbered{hypothesis}{Hypothesis}
\newnumbered{remark}[theorem]{Remark}
\newnumbered{note}{Note}
\newnumbered{observation}{Observation}
\newnumbered{problem}{Problem}
\newnumbered{question}{Question}
\newnumbered{algorithm}{Algorithm}
\newnumbered{example}{Example}
\newunnumbered{notation}{Notation} 

\newcommand{\R}{\mathbb{R}}
\newcommand{\N}{\mathbb{N}}

\newcommand{\tv}{\mathrm{TV}\,}

\newcommand{\BB}{\mathcal{B}}

\newcommand{\DX}{{\Delta x}}
\newcommand{\DT}{{\Delta t}}
\def \ssigma {{\boldsymbol {\sigma} }}
\def \ee {{\rm e}}
\def \cc {{\boldsymbol {c} }}
\def \vv {{\boldsymbol {v} }}
\def \vvv{{\mathbf {v}}}

\title[Decay estimates for the damped wave equation]%
{Decay of approximate solutions for the damped semilinear wave equation 
on a bounded 1d domain}

\author[Amadori, Aqel, Dal Santo]{Debora Amadori, Fatima Al-Zahra' Aqel and Edda Dal Santo}

\classno{35L50, 35B40, 35L20} 

\begin{document}
\maketitle

\begin{abstract}
In this paper we study the long time behavior for a semilinear wave equation
with space-dependent and nonlinear damping term. 
After rewriting the equation as a first order system, we define a class of approximate solutions
that employ tipical tools of hyperbolic systems of conservation laws, such as the Riemann problem.
By recasting the problem as a discrete-time nonhomogeneous system, which is related to a probabilistic interpretation of the solution, 
we provide a strategy to study its long-time behavior uniformly with respect to the mesh size parameter $\DX=1/N\to 0$.
The proof makes use of the Birkhoff decomposition of doubly stochastic matrices and of accurate estimates on the iteration system as $N\to\infty$.

Under appropriate assumptions on the nonlinearity, we prove the exponential convergence in $L^\infty$ of the solution to the first order system
towards a stationary solution, as $t\to+\infty$, as well as uniform error estimates for the approximate solutions. 
\end{abstract}




\textbf{Keywords:} Space-dependent relaxation model, $L^\infty$ error estimate, damped wave equation, initial--boundary value problem
in one dimension.

\section{Introduction}

In this paper we study the initial--boundary value problem for the $2\times 2$ system in one space dimension
\begin{equation}
\begin{cases}
\partial_t\rho +  \partial_x J  = 0, &\\
\partial_t J  +  \partial_x \rho = - 2 k(x) g(J), &
\end{cases} \label{DWE-rho-J-IBVP}
\end{equation}
where $x\in I\, =\, [0,1]$ and $t\ge 0$, and 
\begin{equation}\label{init-boundary-data}
(\rho,J)(\cdot,0) =(\rho_0, J_0)(\cdot)\,,  \qquad \qquad J(0,t)= J(1,t)=J_b
\end{equation}
for $(\rho_0, J_0)\in BV(I)$ and for a constant $J_b \in\R$.  On the function $k=k(x)$ we assume that either 
\begin{equation}\label{hyp-on-k}
k\ge 0\,,\qquad \int_I k(x)\,dx>0
\end{equation}
or the more restrictive assumption 
\begin{equation}\label{hyp:k-unif-positive}
0<k_1\leq k(x)\leq k_{2}\quad \forall\,x\,,\qquad    k_1,\ k_2>0
\end{equation}
hold, while for $g=g(J)$ we require that
\begin{equation}\label{hyp-on-g}
g\in C^1(\R)\,,\qquad g(0)=0\,,\qquad g'(J)>0 \quad \forall\, J\,. 
\end{equation} 
We remark that the assumption \eqref{hyp-on-k} on $k$ includes the possibility of localized damping, 
for instance, $k(x) = \bar k >0$ on some $(\alpha,\beta)$  with $[\alpha,\beta]\subset (0,1)$, and $k(x)=0$ otherwise. In this paper, 
part of the analysis is carried on under assumption \eqref{hyp-on-k}, while for the proof of the main theorem we require that 
$k(x)$ is uniformly positive as in \eqref{hyp:k-unif-positive}. See Remark~\ref{rem:1.1}--{\bf (iv)}.

\smallskip
Problem \eqref{DWE-rho-J-IBVP}--\eqref{init-boundary-data} is related to the one-dimensional damped semilinear wave equation 
on a bounded interval: if $(\rho,J)(x,t)$ is a solution to~\eqref{DWE-rho-J-IBVP}, \eqref{init-boundary-data}, then the function
$$
u(x,t) = J_b t  - \int_0^x \rho(y,t)\,dy
$$
satisfies $u_x=-\rho$, $u_t=J$ and 
\begin{equation}\label{DWE}
\partial_{tt} u - \partial_{xx} u + 2 k(x) g(\partial_t u)=0\,.
\end{equation}
The equation \eqref{DWE} has been considered in several papers, see for instance 
\cite{Haraux78,Haraux88,Haraux-Zua88,Z-1990,CZ94,Haraux09,APT-2017,Cavalcanti-2018}, the review paper
\cite{Z-Sirev-2005} and the recent monograph \cite{Haraux18}. For the homogeneous boundary conditions 
($u=0$ at both ends, corresponding to $J_b=0$), it is well known that the initial-boundary value problem for \eqref{DWE} 
is well-posed for initial data $(u_0,\partial_t u_0) \in H_0^1(I)\times L^2(I)$, for $k(x)\in L^\infty(I)$ with $k(x)\ge 0$, and decay estimates 
for the energy are obtained, either exponential or polynomial. 

\smallskip
Moreover, in \cite{Haraux09},  $L^p$ decay estimates with $2\le p\le \infty$ are studied for the 1-dimensional problem.
These estimates are obtained under the assumption that $g'$ vanishes at $0$, and using the hypotheses of sufficiently regular data,
$(u_0,\partial_t u_0) \in W^{2,\infty}(I)\times W^{1,\infty}(I)$. This regularity restriction appears to be due to the lack of a Lyapunov functional, 
equivalent to the norm of $(u(\cdot,t), u_t(\cdot,t))$ in $W^{1,\infty}(I)\times L^\infty(I)$.

\smallskip
In this paper, we study a very similar problem, assuming that the damping is space-dependent and 
that $g'>0$, see \eqref {hyp-on-g}. Our main contribution is to develop an alternative approach 
that originates from the point of view of the hyperbolic systems of balance laws.
In particular, we construct approximate solutions that allow us to get an accurate description of the solution,
whose evolution is recast as a discrete time system. Then we find a strategy for the analysis of this system, that makes use of a 
discrete representation formula (and not on Lyapunov functionals). This eventually leads to the decay in $L^\infty$ of the solution 
in terms of $(u_x,u_t)$.
Here $u_x(\cdot,t)$, $u_t(\cdot,t))$ belong to $BV(I)\subset L^\infty(I)$ 
so that $(u(\cdot,t), u_t(\cdot,t))$ in $W^{1,\infty}(I)\times L^\infty(I)$; see Subsection~\ref{sec:converg}.


\bigskip
This paper aims at studying the asymptotic properties of the solutions to \eqref{DWE-rho-J-IBVP}--\eqref{init-boundary-data}, 
naturally described by the stationary solutions to  \eqref{DWE-rho-J-IBVP}:
$$
\partial_x J  = 0\,,\qquad \partial_x \rho = - 2 k(x) g(J)\,.
$$
The initial and boundary conditions \eqref{init-boundary-data} lead to a stationary solution $(\widetilde J, \widetilde \rho)$: 
\begin{equation}\label{stationary-sol}
\widetilde J(x)=J_b\,,\qquad   \widetilde \rho(x)  = - 2 g(J_b) \int_0^x k(y)\,dy + C \,,
\end{equation}
the constant $C$ being uniquely identified by the condition
\begin{equation*}
\int_0^1 \widetilde \rho(x)\,dx = \int_0^1 \rho_0(x)\,dx\,,
\end{equation*}
that results in
\begin{equation}\label{def-C-stationary}
C = \int_0^1 \rho_0(x)\,dx + 2 g(J_b) \int_0^1 k(y)(1-y)\,dy\,.
\end{equation}
For the system \eqref{DWE-rho-J-IBVP} a class of approximations of Well-Balanced type to the Cauchy problem 
was studied in \cite{goto,laurent_Book} and in the papers \cite{AG-MCOM16,AG-Briefs15,AG-AnIHP16}. In these last papers, suitable  
$L^1$ error estimates are derived by means of stability analysis for hyperbolic systems of conservation laws, obtained through 
a suitable adaptation of the Bressan-Liu-Yang functional \cite{BLY,Bressan_Book}. 

The same approach to define approximate solutions is adopted in this paper, for the initial-boundary value problem
\eqref{DWE-rho-J-IBVP}--\eqref{init-boundary-data}. We remark that these approximate solutions can be regarded as 
wave-front tracking solutions \cite{Bressan_Book}, with a special choice of the approximate initial data, having discontinuities 
uniformly distributed on a grid.

The analysis performed in this paper is, however, very different from the one for the Cauchy problem. Indeed, the 
semilinear character of system \eqref{DWE-rho-J-IBVP} and the presence of the (reflecting) boundary conditions
lead us to analyze the problem under an unusual perspective:
it can be recasted as the time evolution of the solutions to a finite dimensional linear system, as follows,
\begin{equation}\label{evol-problem}
\ssigma(t^n+)= B(t^n)\ssigma (t^{n-1}+)=B(t^n)B(t^{n-1})\cdots B(0+)\ssigma(0+)\,,
\end{equation}
where $\ssigma(t^n)$ denotes a vector of \emph{wave sizes} appearing in the approximate solution to 
\eqref{DWE-rho-J-IBVP}, \eqref{init-boundary-data} at time $t^n$, while $B(t^n)$ is a doubly stochastic matrix 
(that is, a nonnegative matrix for which the sum of all the elements by row is 1, as well as by column)
that in general depends on time. 
The size of the transition matrices $B$ is $N=1/{\DX}\in 2\N$, where $\DX>0$ represents the mesh size.  

For a review of the properties of non-negative and stochastic matrices, see references \cite{BOOK_BP,Horn-Johnson,Serre-BOOK}.
We refer the reader to Section \ref{sec:iter-matrix} for more details on the derivation of \eqref{evol-problem} and on the structure of $B(t^n)$.
The behaviour of the vector $\ssigma$ is controlled by the spectral properties of the matrix $B$: whenever $g$ is nonlinear 
(that is, $B$ is not constant in time), the behavior of \eqref{evol-problem} is not trivial and may require advanced matrix analysis' tools, 
such as the concept of Joint Spectral Radius (\cite{Jungers09,GuPro13}). 

Also, a possible approach to the study of exponential stability of $\ssigma(t)\equiv0$ in \eqref{evol-problem} 
goes through the existence of a suitable Lyapunov functional. For $N$ fixed it is certainly possible to construct it, 
for instance by constructing a suitable norm on $\R^{2N}$ which is contractive along the discrete trajectories of the system; 
this is possibly done by means of Schur triangularization theorem \cite[Theorem 2.3.1, p.~101]{Horn-Johnson}
and using the fact that the spectral radius of a square matrix $A$ is the greatest lower bound of all the matrix norms of $A$  
\cite[Lemma 5.6.10, p.~347]{Horn-Johnson}. See also the recent preprint \cite{BCS-18_arXiv}.

However, following this strategy, it does not appear clear how to get the needed information on the size of the eigenvalues, uniformly on $N$. 

We overcame this difficulty by working on iterates of $B$  in \eqref{evol-problem} having a constant balance between $n$ and $N$, 
which is the relevant limit. We showed that the discrete representation formula in Theorem~\ref{theo:main} holds, and that the norm 
$\|\cdot\|_{\ell_1}$ on $\R^{2N}$ has a contractive property after a sufficiently large number of iterates $n=2N$ (see Section~\ref{sec:longtime}).


\medskip
We introduce hereafter the main result of this paper. Let $(\rho^\DX,J^\DX)(x,t)$, with $(x,t) \in (0,1)\times [0,\infty)$
denote the approximate solution for \eqref{DWE-rho-J-IBVP}, \eqref{init-boundary-data} defined by the algorithm 
in Section~\ref{sec:approximate}, with $N\in 2\N$, $\DX=1/N$. 
While its precise definition is given in Subsect.~\ref{subsec:approximate_3.1}, we describe here some essential features.
 
Consider the $3\times3$ system
\begin{equation*}
\begin{cases}
\partial_t\rho +  \partial_x J  & = 0\,, \\
\partial_t J  +  \partial_x \rho + 2 g(J)  \partial_x a & =0\,,  \\
\partial_t a &=0\,,
\end{cases}
\end{equation*}
where $a(x) \dot = \int_0^x k(y)\,dy$, and the piecewise constant initial data
\begin{align}\label{approx-data-intro}
\left((\rho_0)^\DX,(J_0)^\DX, a^\DX \right)(x) = \left(\rho_0(x_j+),J_0(x_j+),a(x_j)\right) \,,  ~~ x\in(x_j,x_{j+1})\,, ~~ x_j=j\DX
\end{align}
with boundary condition $J(0,t)=J(1,t) = 0$ (we assume $J_b=0$ for simplicity).  
Then, the function $(\rho^{\DX},J^{\DX},a^{\DX})(x,t)$ is an {\bf exact solution} of the initial-boundary value problem described here above, corresponding to the {\bf approximate initial data} \eqref{approx-data-intro}. The solution is piecewise constant (with respect to space and time), 
and its discontinuities travel with speed $\in \{\pm1,0\}$. 

More precisely, at time $t=0+$ the solution 
is constructed by piecing together the solutions to the local Riemann problems at each $x_j$ and at the boundaries, see Prop.~\ref{prop:1} 
and Figure~\ref{fig:RP}. When two or more discontinuities (which travel with characteristic speed $\in\{\pm1,0\}$)
interact at a positive time, the solution evolves as described in Prop.~\ref{prop:multiple}; see Figure~\ref{fig:multiple}. \\
We remark that a key property is the approximation of the variable $a(x)$ by piecewise constant functions, 
which implies that the effect of source term is concentrated at the points $x_j$ and results in the discontinuities with speed $=0$ in the solution 
to the Riemann problem, see Figure~\ref{fig:RP}. For a detailed definition of $(\rho^{\DX},J^{\DX})$, also in the case $J_b\not=0$,
we refer to Subsect.~\ref{subsec:approximate_3.1}.

\medskip
The main result of this paper here follows. 

\begin{theorem} \label{main-theorem}
Let $g$ satisfy \eqref{hyp-on-g} and $k$ satisfy \eqref{hyp:k-unif-positive}
\begin{equation*}
0<k_1\leq k(x)\leq k_{2}\,, 
\end{equation*}
for some  $k_1$, $k_2>0$. Given $(\rho_0, J_0)\in BV(I)$ and $J_b\in\R$, let $(\widetilde \rho, \widetilde  J)$ be 
the stationary solution as in \eqref{stationary-sol}--\eqref{def-C-stationary}. Define
\begin{equation}\label{d_*}
d_1 = k_1 \min_{J\in D_J} g'(J)> 0 \,,\qquad d_2 = k_2 \max_{J\in D_J} g'(J)
\end{equation}
where $D_J$ is a closed bounded interval depending on the data, which is invariant for $J$. Finally assume that
\begin{equation}\label{cond-on-d_*}
\ee^{2d_2} - 2d_2 < \ee^{2d_1}\,.
\end{equation}
Then there exist constant values $\hat C_j>0$, $j=1,\ldots,5$ that depend only on the coefficients of the equation and on 
the initial and boundary data, 
such that 
\begin{equation}\label{L-infty-decay}
\begin{aligned}
\| J^\DX(\cdot,t)- \widetilde  J\|_{\infty} &\le \hat C_1 {\DX} + \hat C_2 \ee^{-\hat C_3 t}\,,  \qquad    
\\  
\| \rho^\DX(\cdot,t)- \widetilde \rho(\cdot)\|_{\infty} &\le \hat C_4 {\DX} + \hat C_5 \ee^{-\hat C_3 t}\,,
\end{aligned}
\end{equation}
where $\hat C_3$ is given by
\begin{equation*}
\hat C_3 = \frac{1}{2} |\log C(d_1,d_2)|\qquad     C(d_1,d_2) = \ee^{-2d_1}(\ee^{2d_2}-2d_2)\,.
\end{equation*}
\end{theorem}

\begin{remark} \label{rem:1.0}
We observe that the decay estimate holds for the exact solution, as $\DX\to 0$:
\begin{equation*}
\begin{aligned}
\| J(\cdot,t)- \widetilde  J\|_{\infty} &\le \hat C_2 \ee^{-\hat C_3 t}\,,  \qquad    
\\  
\| \rho(\cdot,t)- \widetilde \rho(\cdot)\|_{\infty} &\le  \hat C_5 \ee^{-\hat C_3 t}\,,
\end{aligned}
\end{equation*}
Indeed, this property holds after passing to the limit by means of Helly's theorem, see Subsect.~\ref{sec:converg}.

Hence our result is related to \cite[Theorem 3.1]{Haraux09}, where a decay estimate for the solution of the semilinear wave equation with  
$(u_x(0,\cdot), u_t(0,\cdot))\in W^{1,\infty}(0,\ell)\times W^{1,\infty}(0,\ell)$ is obtained. 

First, some differences occur in the assumptions on the damping term: 
we assume that $g'>0$ while in \cite{Haraux09} the possibly more interesting case of a degeneracy of $g'$ at $J=u_t=0$ is considered; 
we consider a space-dependent damping term, $k(x)g(J)$; finally, we assume some restriction on the nonlinearity, namely \eqref{cond-on-d_*}.

Second, about the regularity of the solutions, our approach can deal with $u_x(\cdot,0),~u_t(\cdot,0) \in BV(I)$ and hence less regular than 
$W^{1,\infty}(0,1)$. It would be interesting to extend the decay estimate to $u_x(\cdot,0),~u_t(\cdot,0) \in L^\infty(I)$ 
and to the case of  $g'$ possibly vanishing as in \cite{Haraux09}.

\end{remark}

\begin{remark} \label{rem:1.1}
Hereby we list several other comments on the Main Theorem~\ref{main-theorem}.

\medskip
{\bf (i)} From \eqref{d_*}, it is clear that $d_1\le d_2$ and that for every $d_1>0$ there exists a non-empty interval of values for $d_2$ for which 
\eqref{cond-on-d_*} holds.

\medskip
{\bf (ii)} If $k(x)\equiv \bar k>0$ and $g'(J)\equiv \bar C>0$ are constant (as in the telegrapher's equation, \cite{kac}) 
then $d_1=d_2 = d$ and then \eqref{cond-on-d_*} is satisfied for every $d=  \bar k \bar C >0$. Moreover one has
\begin{equation*}
\hat C_3 = \frac{1}{2} \left|\log (1-2d \ee^{-2d}) \right| \sim d\qquad \mbox{ as } d\to 0\,.
\end{equation*}

\medskip
{\bf (iii)} For \eqref{cond-on-d_*} to hold, it is necessary that $d_1 >0$ and hence that $g'>0$ as in \eqref{hyp-on-g}. Differently,
if $g'$ vanishes at $J=0$, an exponential decay is no longer expected; see \cite{Haraux09}. 

\medskip
{\bf (iv)} {\sl (About localized damping)} In the main theorem we require that $k$ satisfies the assumption \eqref{hyp:k-unif-positive}; 
in particular $k(x)$ has to be uniformly positive on $(0,1)$.

On the other hand, the construction scheme in Section~\ref{sec:approximate} works under the more general assumption 
\eqref{hyp-on-k} on $k(x)$, that include the case of a localized damping.

\smallskip
The analysis provided in Sect.~4 and Subsect.~5.1, 5.2 is valid under the more general assumption \eqref{hyp-on-k}, while 
\eqref{hyp:k-unif-positive} is used from Subsect.~5.3 on. Under the more restrictive assumption (1.9), the iteration matrices $B(t^n)$ (see \eqref{BB}) 
have a specific structure (all the coefficients $c_j$ are non-zero) 
which allows us to use a very simple Birkhoff decomposition for the matrix corresponding to the linear case (in the sense of (ii) above).
See Proposition~\ref{prop:4.3} and Remark~\ref{rem:inequality-B}.  It would be interesting to extend this analysis for the case of localized damping.    
\end{remark}

The paper is organized as follows. In Section~\ref{Sec:2} we recall some preliminaries on Riemann problems and interaction estimates 
for a $3\times3$ hyperbolic system which is equivalent to \eqref{DWE-rho-J-IBVP}, see \eqref{NC-system}.
In Section \ref{sec:approximate} we describe the Well-Balanced (WB) scheme and in Section \ref{sec:iter-matrix} we introduce the evolution problem 
\eqref{evol-problem} mentioned above, focusing on the spectral properties of the matrix $B$. 

Finally, in the long Section \ref{sec:longtime} we prove Theorem~\ref{main-theorem}, whose proof is outlined at the beginning of the Section. 
The proof is based on a probabilistic interpretation of the solution (see \cite{kac,Dalang} and \cite{Bressan-Shen_2000} for a semilinear hyperbolic system with relaxation),  and on the spectral properties of the evolution problem in \eqref{evol-problem}. 
We use Birkhoff decomposition theorem for doubly stochastic matrices and prove an exponential-type formula in Theorem~\ref{theo:main}.

Thanks to these tools, we first address the linear case as in (ii) above (Subsection~\ref{Subsect_5.2}) and 
prove a contraction property for a norm of the iterated matrix in \eqref{evol-problem} 
(Proposition~\ref{prop:5.8}). Finally we investigate the problem with nonlinear damping in Subsection~\ref{subsec:5.4}, 
where the proof of Theorem~\ref{main-theorem} is presented. 

\section{Preliminaries}\label{Sec:2}
In terms of the diagonal variables $f^\pm$, defined by
\begin{equation}\label{diag-var}
\rho=f^+ + f^-\,,\qquad  J=f^+ - f^-
\end{equation}
the system~(\ref{DWE-rho-J-IBVP}) is rewritten as a discrete-velocity kinetic model
\begin{equation}
\begin{cases}
\partial_t f^- -  \partial_x f^- = {k(x)} \,g(f^+ - f^-),  &
\\
\partial_t f^+ +  \partial_x f^+  = -  {k(x)}\, g(f^+ - f^-) \,. &
\end{cases} \label{GT}
\end{equation}

Now we recall some preliminary results from \cite{AG-MCOM16} dealing with Riemann problems and interaction estimates 
for system \eqref{DWE-rho-J-a}. Our approach is based on an alternative formulation of system \eqref{DWE-rho-J-IBVP} that is obtained 
by adding an equation for the antiderivative of $k$:
\begin{equation}\label{def-a}
a=a(x)\,\dot =\int_{0}^x k(y)\,dy\,, 
\end{equation}
which by \eqref{hyp-on-k} satisfies 
\begin{equation*}
a \in AC(\R) \,, \qquad a_x=k\ge 0\,, \qquad \tv a  = a(1) - a(0) = \|k\|_{L^1} >0\,.
\end{equation*}
This leads to consider the following non-conservative homogeneous $3\times3$ system
\begin{equation}
\begin{cases}
\partial_t\rho +  \partial_x J  & = 0\,, \\
\partial_t J  +  \partial_x \rho + 2 g(J)  \partial_x a & =0\,,  \\
\partial_t a &=0\,,
\end{cases} \label{DWE-rho-J-a}
\end{equation}
which in diagonal variables \eqref{diag-var} is written as
\begin{equation}
\begin{cases}
\partial_t f^- - \partial_x f^- - g(f^+ - f^-)\partial_x a &=0\,, \\
\partial_t f^+ + \partial_x f^+ + g(f^+ - f^-)\partial_x a &=0 \,, \\
\partial_t a &=0\,. 
\end{cases} \label{NC-system}
\end{equation}
Notice that the non-conservative product $g(J)\partial_x a$, which in principle is ambiguous across the discontinuities of $a(x)$,  
is well-defined since $J$ is constant along stationary solutions.

Systems \eqref{DWE-rho-J-a}, \eqref{NC-system} are introduced in order to be able to set up the WB algorithm: this procedure consists 
in localizing a source term of bounded extent into a countable collection of Dirac masses in order to integrate it inside a Riemann solver
by means of an elementary wave, which is obviously linearly degenerate. 
The characteristic speed of system \eqref{NC-system} are $\mp 1, 0$ with corresponding right eigenvectors $(0,1,0)^t$, $(1,0,0)^t$ and $(-g,-g,1)^t$. 
We call \textit{$0$-wave curves} those characteristic curves corresponding to the speed $0$.

\smallskip
In the next Proposition we study the solution to the \textit{Riemann problem} for system~\eqref{DWE-rho-J-a}, that is, the initial value problem
for \eqref{DWE-rho-J-a} with unknown $U=(\rho,J,a)$ and with initial data
\begin{equation}\label{eq:Riemann-init-data}
U(x,0) = \begin{cases}
U_\ell& x<0\\
U_r & x> 0
\end{cases}
\end{equation}
for some constant vectors $$U_\ell= (\rho_\ell,J_\ell,a_\ell)\,, \qquad U_r= (\rho_r,J_r,a_r)\,. $$  
Equivalently, we will denote by $(f^-_\ell,f^+_\ell,a_\ell)$, $(f^-_r,f^+_r,a_r)$ 
the left and right states corresponding to Riemann data to system~\eqref{NC-system}. 


%
\begin{proposition}\cite[Prop.1, p.606]{AG-MCOM16}\label{prop:1}
Assume \eqref{hyp-on-g}.
Let  $m<M$, $a_\ell\le a_r$ and set $\delta ~\dot = ~a_r - a_\ell\ge 0$.

\smallskip

\begin{itemize}
\item[(i)]  The solution to the Riemann problem for system \eqref{DWE-rho-J-a}--\eqref{eq:Riemann-init-data} 
is uniquely determined by
\begin{equation}\label{sol-RP}
U(x,t) = 
\begin{cases}
U_\ell& x/t<-1\\
U_*=(\rho_{*,\ell}, J_*,a_\ell) & -1<x/t<0\\
U_{**}=(\rho_{*,r}, J_*,a_r) & 0<x/t<1\\
U_r & x/t> 1
\end{cases}
\end{equation}
with 
\begin{equation}\label{J*_rho*}
J_*+ g(J_*)\delta = f^+_\ell - f^-_r\,,\qquad \rho_{*,r}-\rho_{*,\ell}= -2g(J_*)\delta\,,
\end{equation}
see Figure \ref{fig:RP}. 

\vspace{3pt}

\item[(ii)] The square $[m,M]^2$ is an \underline{invariant domain} for the Riemann problem projected on the $(f^-,f^+)$-plane. 
This means that if $(f^-_\ell,f^+_\ell)$,  $(f^-_r,f^+_r) \in [m,M]^2$, then the solution $U(x,t)$ given in \eqref{sol-RP} satisfies $(f^-,f^+)(x,t)\in [m,M]^2$.
This property is independent on $\delta\ge 0$.

\vspace{3pt}

\item[(iii)] For every pair $U_\ell$, $U_r$ 
with $(f^-_\ell,f^+_\ell)$,\  $(f^-_r,f^+_r) \in [m,M]^2$, let $\sigma_{-1} = (J_* - J_\ell)$ and  $\sigma_{1} = (J_r-J_*)$.
Hence
\begin{equation}\label{ineq:sizes}
\left| |\sigma_1| - |f^+_r - f^+_\ell|\right| \le C_0 \delta\,,\qquad \left| |\sigma_{-1}| - |f^-_r - f^-_\ell|\right| \le C_0 \delta\,,
\end{equation}
where $C_0 = \max\{g(M-m), - g(m-M)\}$. In particular $C_0$ is independent of $\delta$.
\end{itemize}
\end{proposition}

\begin{remark}\label{rem:invariant-domain-and-damping}
We remark that the invariance domain property stated in Proposition~\ref{prop:1}-(ii) is due to the hypotheses
on the sign of the damping \eqref{hyp-on-k} and \eqref{hyp-on-g}, that is $k(x)\ge 0$, $g(0)=0$ and $g'(J)>0$. 
In other words, if the initial data for $f^\pm$ belong to a square $[m,M]$ and $a_\ell<a_r$, then the intermediate states (their projections $f^\pm$)
belong to the square as well.

Following the proof in \cite[Prop.1, p.606]{AG-MCOM16}, the assumption on $g$ can be slightly weakened; 
indeed it is sufficient that
\begin{equation}\label{cond:monotone-op}
k(x) g(J) \cdot J\ge 0 \qquad \forall\, x,\ J
\end{equation}
therefore including power-like behavior close to the origin $J=0$. We remark that the condition \eqref{cond:monotone-op} guarantees 
the monotonicity of the operator that appears in the abstract formulation of the problem \eqref{DWE-rho-J-IBVP}--\eqref{init-boundary-data}, 
in view of the application of Hille-Yosida theorem (see for instance \cite[Chapt.10]{Brezis-2011}).
\end{remark}


%
%
%

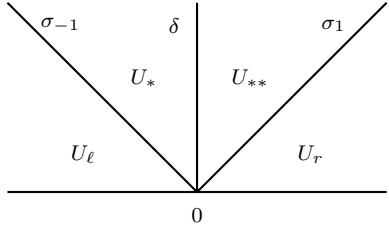
\begin{figure} 
\begin{tikzpicture} 

\draw[thick] (2.5,0.5) -- (2.5,3);
\draw[thick] (0,0.5) -- (5,0.5);
\draw[thick] (2.5,0.5) -- (0,3);
\draw[thick] (2.5,0.5) -- (5,3);

\node () at (1,1) {$U_\ell$};
\node () at (4,1) {$U_r$};
\node () at (1.8,2) {$U_*$};
\node () at (3.2,2) {$U_{**}$};

\node () at (4.3,2.7) {$\sigma_1$};
\node () at (0.7,2.7) {$\sigma_{-1}$};
\node () at (2.2,2.7) {$\delta$};

\node () at (2.5,0.2) {$0$};

\end{tikzpicture}
\caption{The solution to the Riemann problem in Proposition \ref{prop:1}.}\label{fig:RP}
\end{figure}
\begin{figure} 
\scalebox{0.8}
{\begin{tikzpicture} 

\draw[->] (0,0.5) -- (5.5,0.5);
\draw[->] (0.5,0) -- (0.5,5.5);
\draw[dashed,thick, ->] (0.5,0.5) -- (5.5,5.5);
\draw[dashed,thick, ->] (0.5,0.5) -- (-3,3);
\draw[thick] (1,1) -- (5,1);
\draw[thick] (1,1) -- (1,5);
\draw[thick] (1,5) -- (5,5);
\draw[thick] (5,1) -- (5,5);
\draw[dashed, ->] (5,0.5) -- (5,1);
\draw[dashed, ->] (1,0.5) -- (1,1);

\node () at (-2.5,3) {$J$};
\node () at (5.5,5.2) {$\rho$};
\node () at (0.8,5.7) {$f^+$};
\node () at (5.7,0.8) {$f^-$};

\node () at (5,0.2) {$M$};
\node () at (1,0.2) {$m$};
\end{tikzpicture}}
\caption{Invariant domain for systems \eqref{DWE-rho-J-IBVP} and \eqref{GT}}
\label{inv-dom-fig}
\end{figure}
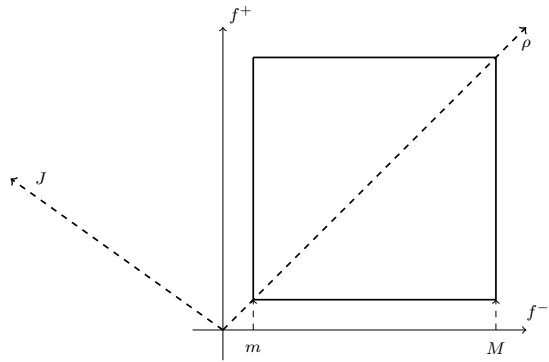

Since the introduction of $a(x)$ yields a nonlinearity, we need to study the interactions of waves in the solutions to \eqref{NC-system}. 
In the notation of Figure \ref{fig:RP}, the amplitude of waves is defined as 
$$
\delta=  a_r-a_\ell
$$ 
for a $0$--wave and 
\begin{align*}
\sigma_{-1} &=  J_* - J_\ell \,=\, -\left( f^-_* - f^-_\ell\right) \,=\,-\left( \rho_* - \rho_\ell \right), \\
\sigma_{1} &=J_r - J_* \,=\, f^+_r - f^+_*  \, = \, \rho_r - \rho_*.
\end{align*}
In other words, if we denote by $\Delta \phi$ the difference $\phi_r -\phi_\ell$ for a certain quantity $\phi$, the sizes $\sigma_{\pm1}$ are given by
\begin{equation}\label{def-sizes}
\sigma_{\pm1} =  \Delta J = \pm \Delta f^\pm = \pm \Delta \rho \,. 
\end{equation}
In particular, we have
\begin{equation}\label{eq:ssigmav-}
\sigma_{1}+ \sigma_{-1} = (J_r - J_*)+ (J_*- J_\ell)=J_r-J_\ell\,.
\end{equation}

The following proposition refines the statement of \cite[Proposition~3]{AG-MCOM16}.

\begin{proposition}[(Multiple interactions)]\label{prop:multiple} 
Assume that at a time $t>0$ an interaction involving a $(+1)$--wave, a $0$--wave and a $(-1)$--wave  occurs, see Figure \ref{fig:multiple}.
Let  $\sigma^-_{-1}$, $\sigma^-_1$ be the sizes of the incoming waves and $\sigma^+_{-1}$, $\sigma^+_1$
be the sizes of the outgoing ones.
Let $\delta = a_r-a_\ell\ge0$ be the size of the  $0$--wave that remains constant across the interaction and assume that 
\begin{equation}\label{A-less-than-1}
(\sup g') \delta <1\,.
\end{equation}
Then, for some $s$ it holds
\begin{equation}\label{mult-inter-matrix-form}
\begin{pmatrix}
\sigma^+_{-1}\\    \sigma^+_1
\end{pmatrix} = \begin{pmatrix}
1-c&c\\
c&1-c
\end{pmatrix} 
\begin{pmatrix}
\sigma^-_{-1}\\
\sigma^-_1
\end{pmatrix},\qquad c=\frac{g'(s)\delta}{g'(s)\delta+1}\,,
\end{equation}
otherwise written as
\begin{equation}\label{mult-inter-component-form}
\begin{aligned}
\sigma^+_{-1} &=  (1-c) \sigma^-_{-1} +c \sigma^-_1\,, \\
\sigma^+_{1} &=  (1-c)\sigma^-_{1} +c \sigma^-_{-1} \,.
\end{aligned}
\end{equation}
Moreover,
\begin{align}\label{eq:no-decay}
|\sigma^+_{-1}| + |\sigma^+_{1}|  &\le~  |\sigma^-_{-1}| + |\sigma^-_{1}|\,,\\
\label{eq:decay}
|\sigma^+_{-1} - \sigma^+_{1}|  &\le~  |\sigma^-_{-1} - \sigma^-_{1}|\cdot \frac{1-\delta (\inf g')}{1+\delta (\inf g')}\,.
\end{align}
\end{proposition}

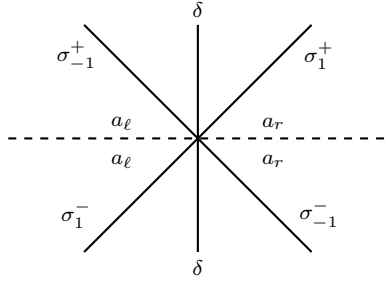
\begin{figure}
\begin{tikzpicture} 

\draw[thick] (2.5,0.5) -- (2.5,3.5);
\draw[thick,dashed] (0,2) -- (5,2);
\draw[thick] (1,0.5) -- (4,3.5);
\draw[thick] (4,0.5) -- (1,3.5);

\node () at (1.5,2.2) {$a_\ell$};
\node () at (3.5,2.2) {$a_r$};
\node () at (1.5,1.7) {$a_\ell$};
\node () at (3.5,1.7) {$a_r$};

\node () at (4.1,3.1) {$\sigma_1^+$};
\node () at (0.9,3.1) {$\sigma_{-1}^+$};

\node () at (0.9,1) {$\sigma_1^-$};
\node () at (4.1,1) {$\sigma_{-1}^-$};

\node () at (2.5,0.3) {$\delta$};
\node () at (2.5,3.7) {$\delta$};

\end{tikzpicture}
\caption{Multiple interaction.}\label{fig:multiple}
\end{figure}

\begin{proof}  Let $J_*^-$, $J_*^+$ be the intermediate values of $J$ before and after the interaction, respectively. 
By \eqref{J*_rho*} these values satisfy 
\begin{equation*}
J^+_* + g(J^+_*) \delta =  f_\ell^+ - f_r^-\,, \qquad \qquad J^-_* - g(J^-_*) \delta =  f_r^+ - f_\ell^-\,.
\end{equation*}
Since the quantity $J_r-J_\ell$ remains constant across the interaction, we get 
$$
J_r-J_\ell = (J_r- J^+_*) + (J^+_*-J_\ell) =  (J_r- J^-_*) + (J^-_*-J_\ell)\,.
$$
Then, by the definition of sizes ($\sigma_{\pm 1} = \Delta J$) we deduce the following identity
\begin{equation}\label{prop:id1}
\sigma^+_1 +  \sigma^+_{-1} = \sigma_1^- +  \sigma^-_{-1} \,.
\end{equation}
The same procedure can be applied to $\rho_r-\rho_\ell$: by \eqref{J*_rho*} and the fact that $\sigma_{\pm 1} = \pm \Delta \rho$, we find  the identity
\begin{equation*}
\sigma^+_1 - \sigma^+_{-1} - 2 g(J^+_*) \delta = \sigma_{1}^- -  \sigma^-_{-1} - 2 g(J^-_*) \delta\,,
\end{equation*}
that can be rewritten as
\begin{align}\nonumber
\sigma^+_1 - \sigma^+_{-1} &= \sigma_{1}^- -  \sigma^-_{-1} + 2\left[ g(J^+_*) - g(J^-_*) \right]\delta\\
&= \sigma_{1}^- - \sigma^-_{-1} + 2 g'(s) \left[ J^+_* - J^-_* \right]\delta  \label{ident-sigma-multi-line-2}
\end{align}
for some $s\in \left(\min\{J_*^+,J_*^-\},\max\{J_*^+,J_*^-\}\right)$.
Notice that 
\begin{equation*}
J^+_* - J^-_* = (J^+_* - J_r) + (J_r - J^-_*) = - \sigma_1^+ + \sigma_{-1}^-
\end{equation*}
and, replacing $J_r$ with $J_\ell$, one has 
\begin{equation*}
J^+_* - J^-_* =  \sigma_{-1}^+ - \sigma_{1}^-\,.
\end{equation*}
Since both equations are true, then one can combine them and write 
\begin{equation*}
J^+_* - J^-_* =  \frac12 \left(\sigma_{-1}^+ - \sigma_{1}^+ + \sigma_{-1}^-  - \sigma_{1}^-\right) \,.
\end{equation*}
By substitution into \eqref{ident-sigma-multi-line-2}, we get
\begin{align*}
\sigma^+_{1} - \sigma^+_{-1} &= \sigma_1^- -  \sigma_{-1}^-  + g'(s)\left(\sigma_{-1}^+ - \sigma_{1}^+ +\sigma_{-1}^- - \sigma_{1}^-\right)\delta\,,
\end{align*}
which leads to 
\begin{equation*}
\left( 1+\delta g'(s) \right) \left(\sigma^+_{1} - \sigma^+_{-1}\right) = \left( 1- \delta g'(s) \right) \left( \sigma_1^- - \sigma_{-1}^- \right)\,.
\end{equation*}
In conclusion, recalling \eqref{prop:id1}, we have the following $2\times2$ linear system
\begin{align}
\sigma^+_1 +  \sigma^+_{-1} & = \sigma_1^- +  \sigma^-_{-1}  \nonumber
\\
\sigma^+_{1} - \sigma^+_{-1} &= \frac{1-g'(s)\delta}{1+g'(s)\delta}  \left( \sigma_1^- -  \sigma_{-1}^- \right)= (1-2c) \left( \sigma_1^- -  \sigma_{-1}^- \right).
 \label{identity:decay}
\end{align}
whose solution is given by \eqref{mult-inter-component-form}, or equivalently by \eqref{mult-inter-matrix-form}.

As for the second part of the proposition, the inequality \eqref{eq:no-decay} follows directly from \eqref{mult-inter-component-form}. 
In order to prove \eqref{eq:decay}, from assumption \eqref{A-less-than-1} and therefore from \eqref{identity:decay} we find
$$
|\sigma^+_{-1} - \sigma^+_1|
\le \frac{1-\delta (\inf g')}{1+\delta (\inf g')} |\sigma^-_{-1} - \sigma^-_{1}|\,. 
$$
This concludes the proof of Proposition.~\ref{prop:multiple}.
\end{proof}

\begin{remark} As a consequence of \eqref{mult-inter-matrix-form}, we can easily check that:
\begin{itemize}
\item the strength of the waves $|\sigma_1| + |\sigma_{-1}|$ remains constant across the interaction when 
$\sigma_{-1}^-\sigma_1^-\ge 0$, that is when the incoming waves have the same sign;
\item on the other hand it decreases strictly whenever  $\sigma_{-1}^-\sigma_1^-<0$, 
leading therefore to a cancellation in terms of the wave strengths.
\end{itemize}
\end{remark}

\section{Approximate solutions}\label{sec:approximate}

In this section we construct WB approximate solutions for the initial--boundary value problem associated to system~\eqref{DWE-rho-J-a}  
(or equivalently \eqref{NC-system}) and initial, boundary conditions \eqref{init-boundary-data} (Subsect.~\ref{subsec:approximate_3.1}), 
study their basic properties (Subsect.~\ref{subsec:approximate_3.2})  and their convergence as $\DX=1/N \to 0$ (Subsect.~\ref{sec:converg}).

\smallskip
To start, we perform the change of variable around the stationary solution $(\widetilde \rho(x),\widetilde J(x)=J_b)$ as in  \eqref{stationary-sol}:
\begin{equation}\label{def:var-around-steady}
v=\rho- \widetilde \rho,\qquad w=J -  J_b\,,
\end{equation}
so that the system \eqref{DWE-rho-J-IBVP}--\eqref{init-boundary-data} rewrites as
\begin{equation}\label{eq:homog}
\begin{cases}
\partial_t v +  \partial_x w  = 0 &\\
\partial_t  w+  \partial_x v = - 2 k(x) \widetilde g (w; J_b) &~~~  \widetilde  g  (w; J_b)=  g( J_b + w) - g(J_b)
\end{cases}
\end{equation}
together with initial-boundary conditions
\begin{equation}\label{init-boundary-data-0}
(v,w)(\cdot,0) =(\rho_0 - \widetilde \rho, J_0 - J_b)(\cdot)\,,  \qquad \qquad w(0,t)= w(1,t)=0
\end{equation}
where $w\mapsto  \widetilde g  (w; J_b)$ has the same properties of $g$ in \eqref{hyp-on-g}, with $\sup g' = \sup \widetilde g'$ 
on corresponding bounded domains, and
\begin{equation}\label{init-data-rho-0}
\int_I v_0\,dx =0\,. 
\end{equation}
From now on we work on the system \eqref{eq:homog}--\eqref{init-boundary-data-0}--\eqref{init-data-rho-0}.
We rename the variables $(v,w)\mapsto (\rho,J)$ and hence assume that
$$
J_b=0\,,\qquad \qquad \int_I \rho_0(x)\,dx=0\,.
$$
Let $D$ be the invariant domain in the $(f^-,f^+)$-variables of Proposition~\ref{prop:1}-(ii),
that is
\begin{equation}\nonumber
D=[\inf_I{f_0^-},\sup_I{f_0^-}]\times[\inf_I{f_0^+},\sup_I{f_0^+}]\,,
\end{equation}
and let 
\begin{equation}\label{D_J-def}
D_J= [J_{{\rm min}}, J_{{\rm max}} ]
\end{equation}
denote the closed interval which is the projection of $D$ on the $J$-axis.
We underline that the invariance of the domain $D$ is due to the "good sign" of the damping term, 
that is $k(x)\ge 0$ and $g'(J)> 0$; see Remark~\ref{rem:invariant-domain-and-damping}.


\subsection{Approximate solutions}\label{subsec:approximate_3.1}
The construction proceeds as in the case of the Cauchy problem (see for instance \cite[p.607]{AG-MCOM16}) 
and is organized into the following steps. See Figure \ref{fig:scheme} for a picture of the scheme for $N=4$.

\medskip\par
\textbf{Step 1: approximation of initial data and of $k(x)$.} Let $N\in 2\N$ be a positive, even number
and set 
\begin{equation*}
  \DX=1/N\,,\qquad \qquad x_j = j\DX\,,\quad j=0,\ldots,N\,.
\end{equation*}
The interval $(0,1)$ is then divided into $N$ cells of length $\DX$, with $x_0=0$ and $x_N=1$.
We approximate the initial data $f_0^\pm$ and $a(x)$ as 
\begin{equation}\label{init-data-approx}
(f_0^\pm)^\DX(x) = f_0^\pm(x_j+)\,,\qquad     a^\DX(x) = a(x_j)\,, \qquad x\in(x_j,x_{j+1})\,. 
\end{equation}
The size of the $0$-wave at a point $0<x_j<1$ is given by
\begin{align}\label{delta-j}
\delta_j = \Delta \{a^\DX\}(x_j) = a(x_j) - a(x_{j-1})= \int_{x_{j-1}}^{x_{j}} k(x) dx\,.
\end{align}
Clearly, we have
\begin{equation}\label{sum-delta_j}
\sum_{j=1}^{N-1} \delta_j = \int_0^{1-\DX} k(x)\,dx \to \|k\|_{L^1}~\mbox{ as }~\DX=\frac 1 N\to 0\,.
\end{equation}
Knowing that $k\in L^1(I)$ and using the absolute continuity of the Lebesgue integral, we can assume $\DX=1/N$ to be sufficiently small so that
\begin{equation}\label{delta-j-small}
C_1 \cdot \delta_j <1\,,\qquad C_1=\sup g'(J)\,,\quad  j=1,\ldots,N-1\,,
\end{equation}
where the supremum is taken over the values of $J$ in the invariant set $D_J$. In this way the assumption \eqref{A-less-than-1} of 
Proposition~\ref{prop:multiple} is satisfied.

For later use, recalling that $\int \rho_0 \,dx =0$ and that $\rho=f^++f^-$, we easily deduce the following inequality:
\begin{equation}\label{int-rho-DX}
\left| \int_I \left[ (f_0^+)^\DX + (f_0^-)^\DX \right]\,dx \right| \le \DX \tv \rho_0\,. 
\end{equation}

\medskip\par
\textbf{Step 2: solution at $t>0$, small $t$.} At $t=0$ each Riemann problem that arises at $0<x_j<1$ is solved using Proposition~\ref{prop:1}.
Moreover, at $x=0$ and $x=1$ we have to deal with two \textit{boundary Riemann problems}. 
For instance, at $x=0$, $t=0$ one has to solve the problem with $(f_0^-,f_0^+) (0+)$ as initial data 
and $J_b=0$ as boundary datum. 
The solution consists of a single $(+1)$-wave and the intermediate state $(f^-_*, f^+_*)$ between $x=0$ and the $(+1)$-wave 
is uniquely determined by
\begin{equation*}
f^-_*= f_0^-\,,\qquad f^+_*- f^-_*= 0 ~~~~\Rightarrow~~~~  f^+_* =  f_0^-\,.
\end{equation*}
The size of the outgoing wave is given by
\begin{equation}\label{size-at-boundary}
\sigma_1= \Delta J = (f_0^+ - f_0^-) = J_0(0+)\,.
\end{equation}    

\medskip\par
\textbf{Step 3: solution at $t>0$, general $t$.} At $t=t^n=n\DT$ with $n\ge 1$, multiple interactions of waves occur at $0<x_j<1$ 
and the newly generated Riemann problems are again solved as in Proposition~\ref{prop:1}.

At $x=0$, let $\sigma_{-1}^-$ be the size of a $(-1)$--wave that hits the boundary. 
Clearly, on the left of this wave the boundary condition  $J_b=0$ is satisfied.
Being $J_r$ the value of $J$ on the right of the incoming wave, its size $\sigma_{-1}^-$ satisfies
$$
\sigma_{-1}^- = \Delta J = J_r\,.
$$
The boundary Riemann problem is solved as before and a new $(+1)$--wave is issued at the point $x=0$, $t=t^n$. 
Since the boundary condition is still satisfied after the interaction, the size of the new wave will be equal to 
\begin{equation}\label{eq:bc}
\sigma_{1}^+ =  \Delta J = J_r  = \sigma_{-1}^- \,.
\end{equation}
Hence the total variation does not change under reflection of waves at the boundaries.
See Figure \ref{fig:boundaryRP} for a picture of this interaction.

\begin{figure}
\begin{tikzpicture} 

\draw[thick] (0,0) -- (0,4);
\draw[thick] (5,0) -- (5,4);

\draw[thick,dashed] (0,2) -- (5,2);
\draw[thick] (1.5,0.5) -- (0,2);
\draw[thick] (3.5,0.5) -- (5,2);
\draw[thick] (0,2) -- (1.5,3.5);
\draw[thick] (5,2) -- (3.5,3.5);

\node () at (4.1,3.3) {$\sigma_{-1}^+$};
\node () at (0.9,3.3) {$\sigma_{1}^+$};

\node () at (0.9,0.7) {$\sigma_{-1}^-$};
\node () at (4.1,0.7) {$\sigma_{1}^-$};

\node () at (0.3,1.3) {$J_b$};
\node () at (0.3,2.7) {$J_b$};
\node () at (4.7,1.3) {$J_b$};
\node () at (4.7,2.7) {$J_b$};

\node () at (1.2,2.2) {$J(0+,t)$};
\node () at (3.8,2.2) {$J(1-,t)$};

\node () at (1.2,1.7) {$J(0+,t)$};
\node () at (3.8,1.7) {$J(1-,t)$};

\end{tikzpicture}

\caption{Interactions with the boundaries $x=0,1$ at time $t>0$. }\label{fig:boundaryRP}
\end{figure}
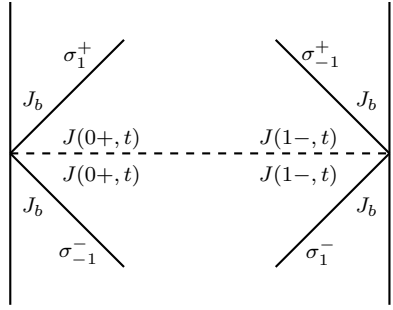

\subsection{Basic properties}\label{subsec:approximate_3.2}
Below we summarize the basic properties of these approximations.

\smallskip
\begin{itemize}
\item \textbf{Invariant domains}. Under the previous construction, the approximate solution attains its values 
in the invariant domain $D$ for every $(x,t)$ as well as the component $J$ is in $D_J$.

\smallskip
\item \textbf{Stationary solutions, stationary approximations}. Recalling \eqref{stationary-sol}, let $\widetilde{J}(x)=J_b\in\R$ and 
$\widetilde \rho(x) = C - 2 g(J_b) a(x)$ be a stationary solution, for some constant $C\in \R$.

In order to be stationary, the approximate initial data \eqref{init-data-approx} must satisfy the boundary condition $J=J_b$ 
and the following relation at $x_j$, $j=1,\ldots,N-1$:
$$
(f_0^\pm)^\DX(x_j+)- (f_0^\pm)^\DX(x_j-)=-g(J_b) \left[ a^\DX(x_j+)- a^\DX(x_j-)\right]\,.
$$
Since $f^\pm = \left(\rho\pm J\right)/2$, it is easy to check that the identity above is valid:
\begin{align*}
(f_0^\pm)^\DX(x_j+)- (f_0^\pm)^\DX(x_j-)&=f_0^\pm(x_j+)- f_0^\pm(x_{j-1}+)\\
&= \frac 12 \left(\widetilde \rho(x_j+)-\widetilde \rho(x_{j-1}+)\right)\\
&= - g(J_b) \left(a(x_j+)- a(x_{j-1}+)\right) \\
&= -g(J_b) \left[ a^\DX(x_j+)- a^\DX(x_j-)\right]\,.
\end{align*}

\smallskip
\item \textbf{Uniform bounds on $\tv(f^\pm)$}. We define 
\begin{align}\label{L-pm}
L_{\pm}(t)&=\sum_{(\pm1)-waves} |\Delta f^\pm| \,, \\ \label{L-0}
L_0(t) &= \frac 12 \left(\sum_{0-waves}  |\Delta f^+|+ |\Delta f^-|  \right)
\end{align}
that by \eqref{def-sizes} are related to $\rho$ and $J$ as 
$$
L_{\pm}(t)=  \tv J(\cdot,t)\,, \qquad \qquad L_{\pm}(t) + L_0(t) =  \tv \rho(\cdot,t)\,.
$$
As in the case of the Cauchy problem \cite{AG-MCOM16}, we claim that $L_\pm(t)$ is not increasing in time
(notice that $L_\pm(t)$ may not tend to zero at $t\to\infty$, uniformly in $N$; see forthcoming Remark~\ref{rem:TV-no-decay}).

Indeed, at time $t\not \in \DT\N$, the quantity $L_\pm(t)$ remains constant, while at $t\in \DT\N$ either it decreases by \eqref{eq:no-decay} 
for interactions inside the domain or it does not change for interactions at the boundary. 
Hence, we obtain that $L_\pm(t)\le L_\pm(0+)$. 
Moreover, using \eqref{ineq:sizes} and \eqref{size-at-boundary}, we have
\begin{align}\nonumber
L_{\pm}(t) \le\, &L_{\pm}(0+)  \\[2mm]
\le\, & \tv f^+(\cdot,0)  + \tv f^-(\cdot,0) + |J_0(0+)| +  |J_0(1-)|  
+ 2 C_0 \tv a \,,\nonumber
\\[2mm]
L_0(t) = &\sum_j |g(J_*(x_j))| \Delta a(x_j) \le C_0 \tv a\,.\nonumber
\end{align}
In conclusion,
\begin{align}\nonumber
& \tv f^+(\cdot,t) + \tv f^-(\cdot,t) =\, L_\pm(t) + 2L_0(t) \\[2mm]
&\qquad  \leq\, \tv f^+(\cdot,0)  + \tv f^-(\cdot,0)   + |J_0(0+)| +  |J_0(1-)|
         + 4\, C_0\, \|k\|_{L^1} \dot = M\,. \label{eq:unif-bound-TV}    
\end{align}
This last inequality provides a bound on the total variation of the solutions which is uniform in $t$ and in $\DX$.
\end{itemize}

\subsection{Strong convergence as $\DX=1/N \to 0$}\label{sec:converg}

It is possible to pass to the limit thanks to Helly's compactness theorem 
(\cite[Theorem 2.4, p.~15]{Bressan_Book} adapted to a bounded interval).
To prove this statement, we observe that the approximate solutions are uniformly bounded (with respect to $t$ and $N=(\DX)^{-1}$) 
in the $L^\infty$--norm and  their total variation is uniformly bounded as well. 
Also, the following property holds: for $M$ defined in \eqref{eq:unif-bound-TV},
\begin{equation}\label{eq:lip-dip-L1}
\int_0^1 |(f^\pm)^\DX (x,t) - (f^\pm)^\DX (x,s)|\,dx\le M |t-s|\qquad \text{for all} ~\DX \quad \text{and}~  t,s\ge 0\,.
\end{equation}
Indeed, let $t$ and $s$ be in the time intervals where no interactions exist, that is 
\begin{equation}\label{time-interval}
 t^n\leq s < t\leq t^{n+\frac 12}\qquad \text{or}\qquad t^{n+\frac 12}\leq s < t\leq t^{n+1}\,,
\end{equation}
then
\begin{align*}
\int_0^1 |(f^\pm)^\DX (x,t) - (f^\pm)^\DX (x,s)|\,dx &=\sum_j |\Delta (f^\pm)^\DX(x_{j} ,\cdot)||\dot{x_j}|~|t-s|\\
&=\sum_{j}|\sigma_{j}(t)|~|t-s|\\
&= \tv (f^{\pm})^\DX(\cdot,t)~|t-s|\leq M ~|t-s|\,,
\end{align*}
where \eqref{eq:unif-bound-TV} is used in the last inequality. Note that for $t$, $s$ in larger intervals than \eqref{time-interval}, 
the map $t\to (f^\pm)^\DX(\cdot,t) \in L^1(0,1)$ is continuous.

Hence, by Helly's theorem \cite[Theorem 2.4]{Bressan_Book}, there exists a subsequence $(\DX)_j\to 0$ such that ${f^\pm}^{(\DX)_j}\to f^\pm$ in $L^1_{loc}(0,1)\times[0,\infty)$
for some functions $f^\pm:(0,1)\times[0,\infty)\to \R$, that are weak solutions of the system~\eqref{NC-system}. 

More precisely, the time-Lipschitz inequality \eqref{eq:lip-dip-L1} is satisfied in the limit as $\DX\to 0$, and hence
functions $f^\pm(x,t)\in L^\infty((0,1)\times [0,\infty))$ are Lipschitz continuous as functions of $t$ in $L^1(0,1)$:
\begin{equation*}
\int_0^1 |f^\pm(x,t) - f^\pm(x,s)|\,dx\le M |t-s|\qquad \text{for all} ~ t,s\ge 0\,.
\end{equation*}
Up to a choice of a representative of $f^\pm$ (the one which is continuous from the right, in space)
one has $f^\pm(\cdot,t)\in BV(I)$, where the function $t\to \tv f^\pm(\cdot,t)$
is non increasing. Also, the $L^\infty$ bounds which are valid for $(f^\pm)^\DX$ are also valid for $f^\pm$; see Remark~\ref{rem:1.0}.

Finally  the equations \eqref{DWE-rho-J-IBVP} for $\rho=f^+ + f^-$\,, $J=f^+ - f^-$ are satisfied in the following sense:

\smallskip
\begin{itemize}
\item[(i)] For all test functions $\phi\in C^1((0,1)\times [0,+\infty))$ one has
\begin{align*}
& \int_0^1\int_{0}^{\infty} \left\{ \rho \partial_t \phi + J \partial_x \phi\right\}\,dxdt + \int_0^1 \rho_0(x) \phi(x,0)\,dx =0\\
& \int_0^1\int_{0}^\infty \left\{ J \partial_t \phi + \rho \partial_x \phi - 2k(x) g(J)\right\}\,dxdt + \int_0^1 J_0(x) \phi(x,0)\,dx =0
\end{align*}
\item[(ii)] $J(0,t) = J_b = J(1,t)$  for a.e. $t>0$\,.
\end{itemize}

\medskip
Following the analysis in \cite{A-1997} for the non-characteristic initial-boundary value problem, one could prove that 
the boundary condition (ii) is attained for every $t>0$ except at most countably many.

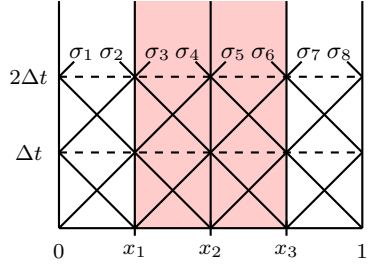
\begin{figure}
\begin{tikzpicture} 

\fill[red!20!white] (1,0) -- (3,0) -- (3,3) -- (1,3) -- (1,0);

\draw[thick] (0,0) -- (0,3);
\draw[thick] (4,0) -- (4,3);
\draw[thick] (0,0) -- (4,0);

\draw[thick] (1,0) -- (1,3);
\draw[thick] (3,0) -- (3,3);
\draw[thick] (2,0) -- (2,3);

\draw[thick,dashed] (0,1) -- (4,1);
\node () at (-0.4,1) {$\DT$};

\draw[thick,dashed] (0,2) -- (4,2);
\node () at (-0.4,2) {$2\DT$};

\draw[thick] (0,0) -- (1,1);
\node () at (0.3,2.3) {\small{$\sigma_1$}};

\draw[thick] (1,0) -- (0,1);
\draw[thick] (1,0) -- (2,1);
\node () at (0.7,2.3) {\small{$\sigma_2$}};
\node () at (1.3,2.3) {\small{$\sigma_3$}};

\draw[thick] (2,0) -- (1,1);
\draw[thick] (2,0) -- (3,1);
\node () at (1.7,2.3) {\small{$\sigma_4$}};
\node () at (2.3,2.3) {\small{$\sigma_5$}};

\draw[thick] (3,0) -- (2,1);
\draw[thick] (3,0) -- (4,1);
\node () at (2.7,2.3) {\small{$\sigma_6$}};
\node () at (3.3,2.3) {\small{$\sigma_7$}};

\draw[thick] (4,0) -- (3,1);
\node () at (3.7,2.3) {\small{$\sigma_8$}};

\draw[thick] (0,1) -- (1,2);

\draw[thick] (1,1) -- (0,2);
\draw[thick] (1,1) -- (2,2);

\draw[thick] (2,1) -- (1,2);
\draw[thick] (2,1) -- (3,2);

\draw[thick] (3,1) -- (2,2);
\draw[thick] (3,1) -- (4,2);

\draw[thick] (4,1) -- (3,2);

\draw[thick] (1,-0.1) -- (1,0.1);
\node () at (1,-0.3) {$x_1$};

\draw[thick] (2,-0.1) -- (2,0.1);
\node () at (2,-0.3) {$x_2$};

\draw[thick] (3,-0.1) -- (3,0.1);
\node () at (3,-0.3) {$x_3$};

\node () at (0,-0.3) {$0$};
\node () at (4,-0.3) {$1$};

\draw[thick] (0,2) -- (0.2,2.2);

\draw[thick] (1,2) -- (0.8,2.2);
\draw[thick] (1,2) -- (1.2,2.2);

\draw[thick] (2,2) -- (1.8,2.2);
\draw[thick] (2,2) -- (2.2,2.2);

\draw[thick] (3,2) -- (2.8,2.2);
\draw[thick] (3,2) -- (3.2,2.2);

\draw[thick] (4,2) -- (3.8,2.2);

\end{tikzpicture}
\caption{Well-balanced scheme in the case $N=4$.}\label{fig:scheme}
\end{figure}

\section{The iteration matrix}\label{sec:iter-matrix}
In this section we describe our strategy to study the long-time behavior of the approximate solutions. Let 
\begin{equation*}
\ssigma(t)=\left(\sigma_1,\ldots,\sigma_{2N}\right) \in \R^{2N}\,,\qquad N\in2\N
\end{equation*}
be the vector of the sizes of the waves which are present in the solution at time $t$, ordered according to
increasing space position, and denote their location by
$$
y_1(t) < y_2(t) < \ldots < y_{2N}(t)\qquad \forall \, t >0\,,\ t\not = t^n,\ t\not =t^{n+1/2}\,. 
$$
To study the evolution in time of the vector $\ssigma$, we make iterative use of Proposition~\ref{prop:multiple}. 
An important role is played by the \emph{transition coefficients} $c=c^n_j$ that appear 
in \eqref{mult-inter-matrix-form} and correspond to a single interaction at time $t^n$ and $x=x_j$, that is:
\begin{equation}\label{def:c_j}
c^n_j = \frac{g'(s_j^n) \delta_j}{g'(s_j^n) \delta_j+1} \ge 0 \,, \qquad s_j^n\in D_J\,,\qquad 
j=1,\ldots,N-1\,,\quad n\ge 1, 
\end{equation}
where $\delta_j$ is given in \eqref{delta-j}, $D_J$ in \eqref{D_J-def} and $s_j^n$ depends on the solution. We define
\begin{equation}\label{def:cc}
\cc=\cc^n=(c^n_1,\ldots,c^n_{N-1})\in\R^{N-1}\,.
\end{equation}
In the following we will often drop the index $n$ when the time $t=t^n$ is fixed and write $c_j$ in place of $c^n_j$, 
so that we denote $\cc=(c_1,\ldots,c_{N-1})$. 

We remark that the map
$$
D_J^{N-1}\ni (J_1,\ldots,J_{N-1}) \mapsto \cc = \left( \frac{g'(J_1) \delta_1}{g'(J_1) \delta_1+1}\,,\ldots, 
\frac{g'(J_{N-1}) \delta_{N-1}}{g'(J_{N-1}) \delta_{N-1}+1}\right)
$$
is continuous over the compact set $D_J^{N-1}\subset \R^{N-1}$, hence its image is a compact set $K\subset\R^{N-1}$, 
which is the set of all the possible values of the vectors $\cc$. 

By the smallness of $\delta_j$ (see \eqref{delta-j} and \eqref{delta-j-small}) we have that
\begin{equation}\label{c_j-bound}
\frac{\inf g'}2 \delta_j\le c^n_j \le \min\{C_1\delta_j, 1/2 \}\,, \qquad j=1,\ldots,N-1\,.
\end{equation}
Let us give an estimate on the $\ell_1$-norm of $\cc^n$, being $\|\cc^n\|_{1}= \sum_{j=1} ^{N-1} c^n_j$.
Recalling \eqref{sum-delta_j} and \eqref{c_j-bound}, we immediately get  
\begin{equation}\label{bound-on-sum-c_j}
\frac{\inf g'}2   \int_0^{1-\DX} k(x)\,dx   \le   
\|\cc^n\|_{1} \le C_1  \|k\|_{L^1} \,.
\end{equation}
In the next lemma we relate the iteration step to a suitable \emph{transition matrix} $B$.

\begin{lemma}\label{lem:BB}
At time $t^n=n\DT$ the vector $\ssigma$ evolves according to 
\begin{equation}\label{eq:evol-of-sigma}
\ssigma(t^n+)= B(\cc) \,\ssigma (t^{n-1}+)\,,\qquad  n\ge 1
\end{equation}
where $B(\cc) \in \R^{2N \times 2N}$ is 
\begin{equation}\label{BB}
{B}(\cc)= \begin{bmatrix}
0&  1  & 0 & 0&\cdots&0 &0&0&0\\
c_1 &0 &0&1-c_1&\cdots&0&0&0&0\\
1-c_1&0& 0 & c_1  & &\vdots &\vdots \\
\vdots&\vdots& &\vdots &\vdots &\vdots&\vdots \\
0&0&0&0&\cdots& c_{N-1} &0&0 & 1-c_{N-1}\\
0&0&0&0&\cdots&1-c_{N-1}& 0 &0 & c_{N-1}\\
0&0&0&0&\cdots&0&0&1&0
\end{bmatrix}
\end{equation}
which is doubly stochastic
\footnote{A doubly stochastic matrix is a nonnegative matrix for which the sum of all the elements by row is 1, as well as by column}.  
The following properties hold:
\begin{itemize}
\item[(i)]  The determinant of $B$ is 
\begin{equation}\label{det-of-B-formula}
\det(B) = - \left(1-2c_1\right) \cdots  \left(1-2c_{N-1}\right)\,.
\end{equation}

\item[(ii)] The eigenvalues $\lambda_i$ of $B$ satisfy $|\lambda_i|\le 1$ for all $i=1,\ldots,2N$;

\item[(iii)] The values $\lambda=\pm 1$ are eigenvalues with corresponding (left and right) eigenvectors
\begin{align}\label{v-pm}
\begin{aligned}
\lambda_-= -1\,,\qquad & v_{-}=(1,-1,-1,1,\ldots,1,-1,-1,1)\,,\\
\lambda_+= 1\,,\qquad & e=(1,1,\ldots,1,1)\,.
\end{aligned}
\end{align}

\item[(iv)] If 
\begin{equation}\label{hyp:loop2}
c_j \cdot c_{j+1}>0\qquad \mbox{ for some }j\,, 
\end{equation}
that is, if there are two consecutive coefficients that do not vanish, 
then the eigenvalues with maximum modulus are exactly two ($\lambda=\pm 1$) and they are simple. 
\end{itemize}
\end{lemma}

\begin{proof}
The construction is divided into three steps. 

\smallskip
\textbf{1.} At time $t=(n-\frac12)\DT$, $n\ge 1$, each pair of components $\sigma_{2i-1}$ and $\sigma_{2i}$ are switched, $i=1,\ldots,N$.
In matrix form, one has the permutation
\begin{equation}\label{B1}
\ssigma(t+) = B_1 \ssigma(t-)\,,\quad B_1\doteq \begin{bmatrix}
0&1&0&\cdots&0&0\\
1&0&0&\cdots&0&0\\
\vdots&\vdots& \ddots & &\vdots &\vdots \\
\vdots&\vdots& &\ddots &\vdots &\vdots \\
0&0&0&\cdots&0&1\\
0&0&0&\cdots&1&0
\end{bmatrix}.
\end{equation}

\smallskip
 \textbf{2.} At time $t=n\DT$, by \eqref{mult-inter-matrix-form} we have
\begin{equation}\label{B2}
\ssigma(t+) = B_2\ssigma(t-)\,, \quad 
B_2(\cc)  = \begin{bmatrix}
1&  0  & 0 &\cdots &0&0&0\\
0&c_1 & {1-c_1}&\cdots&0&0&0\\
0&{1-c_1}& c_1 &  & &\vdots &\vdots \\
\vdots&\vdots& &\ddots &\vdots &\vdots&\vdots \\
0&0&0&\cdots&c_{N-1} &1-c_{N-1}&0\\
0&0&0&\cdots&1-c_{N-1} & c_{N-1} &0\\
0&0&0&\cdots&0&0&1
\end{bmatrix}.
\end{equation}

\smallskip
\textbf{3.} Finally we write 
\begin{equation}\label{B-of-cc}
B(\cc)  \doteq B_2(\cc)  B_1
\end{equation}
and obtain \eqref{BB}.
 
\paragraph{Proof of (i).} By the Binet Theorem, see \cite[p. 28]{Horn-Johnson}, we have
\begin{equation*}
\det(B) = \det(B_2)  \det(B_1)
\end{equation*}
where 
\begin{equation*}
\det(B_1) = 1\,,\qquad  \det( B_2)  = (2c_1 - 1) \cdots (2c_{N-1} - 1)\,. 
\end{equation*}
Since $(N-1)$ is odd, we obtain \eqref{det-of-B-formula}\,. 

\paragraph{Proof of {\it (ii)} and  {\it (iii)}.} By Gershgorin Theorem, see \cite[p. 387]{Horn-Johnson}, all the eigenvalues of the matrix $B$ are located 
in the circle of center $0$ and radius $1$ in the complex plane. Indeed, all the terms on the diagonal are 0 and 
$$
\sum_{j=1\,,i\not = j}^{2N} |B_{ij}|=1\,,\qquad \forall\, i\,.
$$
Hence  {\it (ii)} follows. About {\it (iii)} it is immediate to check that 
$$
B v_{-}=- v_{-}\,, \qquad  v_{-}^t   B =- v_{-}^t 
$$
while $Be=e$ and $e^t B = e^t$ follow by the double stochastic character of $B$.

\paragraph{Proof of {\it (iv)}.} It remains to prove that $\lambda_{\pm}$ are the {\sl only} eigenvalues of $B$ with modulus 1, 
while all the other have modulus $<1$. 

We claim that $B$ satisfies the hypotheses of Romanovsky Theorem, see \cite{Romanovsky} and  \cite[p. 541]{Horn-Johnson}.  
The latter result states that a nonnegative irreducible matrix $A\in M_n(\R)$ has exactly $p\in\N$ eigenvalues with maximum modulus if, 
for any node of the corresponding directed graph, $p$ is the greatest common divisor of the lengths of all the directed paths 
that both start and end at a same node.

See Figure \ref{fig:1} for a picture of the graph related to the matrix ${B}=[{B}_{ij}]_{i,j=1,\dots 2N}$, 
where each node correspond to a row $i$ and each directed arc $(i,j)$ corresponds to a non-zero element ${B}_{ij}$. 
Remark that the graph of ${B}$ can be deduced by noticing that the first row is represented by the arc $(1,2)$, 
the last row by the arc $(2N,2N-1)$ and that each $2\times 4$ submatrix occupying the block of rows $2j,2j+1$ and columns $2j-1,\dots,2j+2$,
\begin{equation*}
{\hat B}_j=\begin{bmatrix}
c_j &0 &0&1-c_j\\
1-c_j&0& 0 & c_j
\end{bmatrix} \qquad j=1,\dots, N-1,
\end{equation*}
corresponds to a squared subgraph made of the arcs $(2j,2j-1)$, $(2j,2j+2)$, $(2j+1,2j-1)$, $(2j+1,2j+2)$.
Notice that, if $c_j=0$, then only the upper arc $(2j,2j+2)$ and the lower one $(2j+1,2j-1)$ survive in the squared subgraph related to ${\hat B}_j$.
The whole graph is then obtained by juxtaposing the arcs $(1,2)$, $(2N,2N-1)$ to the subgraphs representing ${\hat B}_j$, for $j=1,\dots,N-1$.


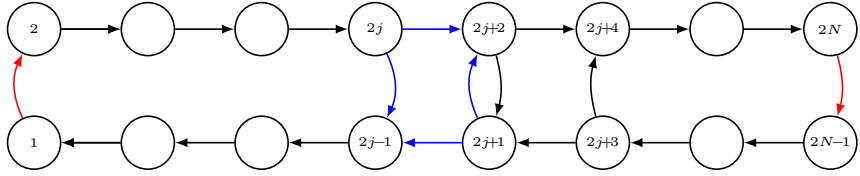
\begin{figure}

\begin {tikzpicture}[-latex ,auto, node distance =1.5 cm and 1.5cm ,on grid ,
semithick ,
state/.style ={ circle ,draw , minimum width =0.7 cm}]
\node[state] (C) {\tiny{$1$}};
\node[state] (A) [above =of C] {\tiny{$2$}}; 
\node[state] (B) [above right =of C] {}; 
\node[state] (D) [right =of C] {}; 
\node[state] (E) [right =of B] {}; 
\node[state] (F) [right =of E] {\tiny{$2j$}}; 
\node[state] (G) [right =of F] {\tiny{$\hspace{-3pt}2j\hspace{-3pt}+\hspace{-3pt}2\hspace{-3pt}$}}; 
\node[state] (H) [right =of G] {\tiny{$\hspace{-3pt}2j\hspace{-3pt}+\hspace{-3pt}4\hspace{-3pt}$}}; 

\node[state] (I) [right =of D] {}; 
\node[state] (J) [right =of I] {\tiny{$\hspace{-3pt}2j\hspace{-3pt}-\hspace{-3pt}1\hspace{-3pt}$}}; 
\node[state] (K) [right =of J] {\tiny{$\hspace{-3pt}2j\hspace{-3pt}+\hspace{-3pt}1\hspace{-3pt}$}}; 
\node[state] (L) [right =of K] {\tiny{$\hspace{-3pt}2j\hspace{-3pt}+\hspace{-3pt}3\hspace{-3pt}$}}; 

\node[state] (M) [right =of H] {}; 
 
\node[state] (N) [right =of L] {}; 

\node[state] (O) [right =of N] {\tiny{$\hspace{-3pt}2N\hspace{-3pt}-\hspace{-3pt}1\hspace{-3pt}$}}; 

\node[state] (Q) [right =of M] {\tiny{$2N$}}; 

\path[color=red] (C) edge [bend left =25] node[below =0.15 cm] {} (A);

\path[color=blue] (K) edge [bend left =25] node[below =0.15 cm] {} (G);
\path (G) edge [bend right = -15] node[below =0.15 cm] {} (K);



\path[color=blue] (F) edge [bend left =25] node[below =0.15 cm] {} (J);

\path (L) edge [bend right = -15] node[below =0.15 cm] {} (H);


\path[color=red] (Q) edge [bend right = -15] node[below =0.15 cm] {} (O);

\path (A) edge node {} (B);
\path (D) edge node {} (C);

\path (B) edge node {} (E);
\path (I) edge node {} (D);

\path (A) edge node {} (B);
\path (D) edge node {} (C);

\path (E) edge node {} (F);
\path (J) edge node {} (I);

\path[color=blue] (K) edge node {} (J);
\path[color=blue] (F) edge node {} (G);

\path (G) edge node {} (H);
\path (L) edge node {} (K);

\path (H) edge node {} (M);
\path (N) edge node {} (L);

\path (O) edge node {} (N);
\path (M) edge node {} (Q);

\end{tikzpicture}
\caption{\label{fig:1} The graph corresponding to ${B}$ when $c_j$, $c_{j+1}>0$. 
The red arcs correspond to the first and final row of the matrix, while the blue arcs connecting the nodes $2j-1,2j,2j+1,2j+2$ 
correspond to the submatrix ${\hat B}_j$.}
\end{figure}


First, notice that $B$ is irreducible, which is equivalent to say that the graph is totally connected, namely that each node can be reached 
from any other node via a path made of arcs present in the graph: this holds true since one can always follow the circuit 
$(1,2,4,\dots,2j,2j+2,\dots,2N,2N-1,\dots, 2j+1,2j-1,\dots,3,1)$ from any node in the graph.
Secondly, the length of any path in the graph connecting a node to itself can be divided at most by $2$, which means that in this case $p=2$.
Indeed, there is no way to obtain a path of odd length because there are no diagonal arcs. 
Moreover, by assumption there exists an index $j$ such that $c_j$, $c_{j+1}$ are not zero as in Figure \ref{fig:1}.

Then, it is easy to see that there are at least two  
paths connecting the node $1$ to itself of lengths 
$2j$ and $2j+2$ and the great common divisor must be $2$.

Now, by the Romanovsky Theorem we can conclude that $\lambda_{\pm}$ are the only two eigenvalues with modulus $1$ and 
the proof of (iv) is complete.
\end{proof}

\begin{remark}\label{rem:B-dep-on-t} 
Notice that in general $B_2$ depends on $t^n$, since the coefficients $c_j$ depend on $g'(J)$. 
However, the structure of the matrix $B$ (the coefficients which are $\not =0$) 
does not change with $n$, in the sense that, for a fixed $j$, either $c^n_j\not=0 $ for every $n$ or $c^n_j =0 $ for every $n$.
\end{remark}

It is well known that doubly stochastic matrices can be written as a convex combination of permutations 
by Birkhoff Theorem (\cite[Theorem 8.7.2]{Horn-Johnson}). In the next proposition, for $\cc$ constant 
we give an explicit Birkhoff decomposition of the matrix $B(\cc)$.

\begin{proposition}\label{prop:4.3}
	Let $\cc=c\, (1,\ldots,1)\in\R^{N-1}\,$, for some constant $c\in [0,1/2)$. Then the matrix $B$ can be decomposed as	
	\begin{equation}\label{decompose-B}
		B(\cc)=(1-c)B(0)+cB_1.
	\end{equation}
\end{proposition}
\begin{proof}
	Since $\cc$ is  constant, then the matrix $B_2(\cc)$ in \eqref{B2} can be written as
	\begin{equation}\label{decompose-B2}
		B_2(\cc)=(1-c)B_2(0)+cI.
	\end{equation}
	Recalling that  $B(\cc)  = B_2(\cc)  B_1$ and substituting \eqref{decompose-B2}, we obtain \eqref{decompose-B}. 
\end{proof}

\begin{remark}\label{rem:inequality-B} 
Assume that  \eqref{hyp:k-unif-positive} holds, that is $0<k_1\leq k(x)\leq k_{2}$ for some positive $k_1,k_2$. 
Hence, see \eqref{delta-j}, $\delta_j$ is bounded as
$$
\frac{k_1}N \le \delta_j \le \frac{k_2}N\,.
$$
Let us define $d_1$, $d_2$ as in \eqref{d_*}, that is 
\begin{equation*}
d_1 = k_1 \min_{J\in D_J} g'(J)> 0 \,,\qquad d_2 = k_2 \max_{J\in D_J} g'(J)\,.
\end{equation*}
By the monotonicity of the map $x\to \frac x {x+1}$, the bounds in \eqref{c_j-bound} become: 
\begin{equation}\label{c_j-bound-k-bounded}
\frac{d_1/N}{1+ d_1/N} \leq c^n_j \leq\frac{d_2/N}{1+d_2/N}\,.
\end{equation}
Hence
\begin{equation*} 
B(\cc^{n})\leq \left(1-\frac{d_1/N}{1+ d_1/N} \right)B(0)+ \frac{d_2/N}{1+d_2/N} B_{1}\,,
\end{equation*}
and after simple passages, it is rewritten as
\begin{equation}\label{inequality-B} 
B(\cc^{n})\leq \left(1+ \frac{d_1}N \right)^{-1} \left[B(0) + \frac{d_2}N B_{1} \right]\,.
\end{equation}
Note that the inequality in \eqref{inequality-B} is an entrywise inequality.
\end{remark}

\section{Long time behaviour of the approximate solutions}\label{sec:longtime}

In this section we study the behaviour of $\ssigma(t^n)$ as $n\to +\infty$ (i.e.\ as $t\to +\infty$)
and as $N\to\infty$ ($\DX\to 0$). The main results are listed here below, each item corresponding to a subsection.

\smallskip
\begin{itemize}

\item[{\bf (1)}] Proposition~\ref{prop:5.2} relates the $L^\infty$-norm of $J(\cdot,t^n)$, $\rho(\cdot,t^n)$ as $n\to\infty$ to the evolution of the $\ell_1$--norm 
of the operator $\BB_n$
\begin{equation}\label{def:BB}
\BB_n \dot= \left[ B^{(n)} B^{(n-1)}\cdots  B^{(2)}  B^{(1)}\right]\,,\quad B^{(n)} = B(\cc^n)\,,\qquad n\in \N
\end{equation}
on the eigenspace 
\begin{equation}\label{E_-}
E_-\dot =<e,v_->^\perp\,.
\end{equation}

\item[{\bf (2)}] Lemma~\ref{lem:decomp} concerns a convenient decomposition of the vectors in $E_-$, along which a suitable cancellation occurs later on.

\item[{\bf (3)}] In Theorem~\ref{theo:main}, the exponential formula $\left[B(0)+ {\frac d N}  B_1 \right]^{2N} \in M_{2N}$ is estimated in terms of 
$d$ and $N$,  the difficulty lying in the fact that the matrices $B(0)$ and $B_1$ do not commute. 
The proof relies on a detailed study of the expansion of the power whose coefficients are described by hypergeometric functions, and their sum 
is computed through modified Bessel functions.\\
Thanks to a careful expression of the first order in $1/N$, a cancellation property is identified (see Proposition~\ref{prop:5.8}). As a result,
it is found that the $\||\BB_{2N}\||_1 < 1$ on $E_-$, where
\begin{equation*}
\|| A \||_1  = \max_{j} \sum_{i=1}^n |a_{ij}|\,,\qquad A=(a_{ij})\in M_{n}
\end{equation*}
is the maximum column sum matrix norm, which is induced by the $\ell_1$-norm on $\R^n$.

\item[{\bf (4)}] Finally, in Subsection~\ref{subsec:5.4}, we combine the previous results and prove Theorem~\ref{main-theorem},
starting from the inequality \eqref{inequality-B} which is obtained by a Birkhoff decomposition of the generic matrix $B(\cc)$.
\end{itemize}

\subsection{A first decomposition of the strength vector}\label{subsec:5.1}
We decompose the initial vector $\ssigma(0+)$ as follows:
\begin{equation*}
\ssigma(0+)=\frac{(\ssigma(0+)\cdot e)}{2N}\,e+\frac{(\ssigma(0+)\cdot v_-)}{2N}\,v_-+ \widetilde{\ssigma}(0+),
\end{equation*}
where $e, \,v_-$ are the eigenvectors defined at \eqref{v-pm}  and $\widetilde{\ssigma}(0+) \in {E_-}$.

As a consequence of the boundary conditions $J(1-,t) = J(0+,t)=0$, we get
\begin{align*}
\ssigma(0+)  \cdot e=\sum_{j=1}^{2N} \sigma^{0}_j  = \sum {\Delta J}(x_j,0+) = J(1-,0+) - J(0+,0+)= 0.
\end{align*}
Hence the decomposition of $\ssigma(0+)$ reduces to
\begin{equation}\label{eq:decomp-0}
\ssigma(0+)=\frac{(\ssigma(0+)\cdot v_-)}{2N}\,v_-+ \widetilde{\ssigma}(0+).
\end{equation}
Consider the matrix $\BB_n$ defined at \eqref{def:BB},
obtained by iterating the step \eqref{eq:evol-of-sigma}.  By means of \eqref{eq:decomp-0} and using again \eqref{v-pm} for $v_-$, we get that 
\begin{align}
\ssigma(t^n+) &=\BB_n \ssigma(0+) 
=(-1)^n  \frac{(\ssigma(0+)\cdot v_-)}{2N}\,v_- + \BB_n \widetilde{\ssigma}(0+)\,. \label{eq:ssigma1}
\end{align}

In the following proposition we employ \eqref{eq:ssigma1} to obtain $L^\infty$-bounds on $J=J^\DX$, $\rho=\rho^\DX$.
First, let us define the extended initial data  $\bar J_0:[0,1]\to\R$,
\begin{equation}\label{def:bar_J0}
\bar J_0(x) = \begin{cases}J_0(x) & 0<x<1\\ 0 & x=0 \mbox { or } 1\,.   \end{cases}
\end{equation}
It is clear that $\tv \bar J_0 = \tv \{\bar J_0;[0,1]\} =  |J_0(0+)|+  \tv \{ J_0; (0,1)\} + |J_0(1-)|$.

\begin{proposition}\label{prop:5.2}
For every $t\in(t^n,t^{n+1})$ one has 
\begin{align}\label{bound-on-sup-J}
\| J(\cdot,t)\|_{\infty} &\le \frac{1}{2N} \tv \bar J_0 
+ \|\BB_n  \widetilde{\ssigma}(0+)\|_{\ell^1}     
\\   \label{bound-on-sup-rho}
\| \rho(\cdot,t)\|_{\infty} &\le \frac 2 N (1+C_1\|k\|_{L^1}) \,  \tv \bar J_0   
~+~ 2 (1+2 C_1\|k\|_{L^1})  \|\BB_n  \widetilde{\ssigma}(0+)\|_{\ell^1} ~+~ \frac 1 N \tv \rho_0 \,.
\end{align}
\end{proposition}

\begin{proof} We start by observing that the following inequality holds,
\begin{equation}\label{bound-on-comp-sigma-v}
\left| \ssigma(0+)\cdot v_- \right| \le   \tv \bar J_0 \,.  
\end{equation}
Indeed, by recalling the definition of $v_-$ in \eqref{v-pm}, we observe that
\begin{align*}
\ssigma(0+)\cdot v_- 
&=\, \sigma_1^0 + \sum_{j=1}^{N-1} (-1)^{j} \left(\sigma_{2j}^0 + \sigma_{2j+1}^0  \right) +\sigma_{2N}^0\,.
\end{align*}
Recalling that $\sigma^0_{2j}$, $\sigma^0_{2j+1}$ are the two outgoing waves at $x_j=j\DX$ and time $t=0$, then by \eqref{eq:ssigmav-} it holds
$$
\sigma_{2j}^0 + \sigma_{2j+1}^0 = J(x_j+,0)-J(x_j-,0)\,.  
$$
Moreover, since the approximate solution satisfies the boundary conditions $J=0$, for small $t$ we have 
$$
\sigma_1^0 = J(x_1-,0)-J(0+,t)= J(x_1-,0)= J(0+,0)\,,\qquad \sigma_{2N}^0 = - J(1-,0)\,.
$$
Therefore,
\begin{align}\label{eq:ssigma0v-}
\ssigma(0+)\cdot v_- & =\, J(x_1-,0) +\sum_{j=1}^{N-1}(-1)^j(J(x_j+,0)-J(x_j-,0)) -J(x_{N-1}+,0)
\end{align}
and then, by recalling \eqref{init-data-approx}, we find that
\begin{align*}
\left|\ssigma(0+)\cdot v_-\right|\le\, 
& |J_0(0+)|+  \tv J_0 + |J_0(1-)|
\end{align*}
that gives \eqref{bound-on-comp-sigma-v}.

\smallskip
\textbf{Proof of \eqref{bound-on-sup-J}.}\quad 
Let $y_\ell(t)$ denote the location of a $\pm 1$-wave at time $t$, for $\ell=0,\dots,2N$.
Observe that, for every $x\not = y_\ell$, the value of $J(x,t^n+)$ is expressed by a partial sum of the $\sigma^n_\ell$:
$$
J(x,t^n+) = \underbrace{J(0+,t^n+)}_{=0}   + \sum_{y_\ell < x} \Delta J(y_\ell,t^n+) =   \sum_{y_\ell < x} \sigma^n_\ell = \ssigma(t^n+) \cdot \vv
$$
where 
\begin{equation}\label{eq:vv}
\vv=(v_1,\ldots,v_{2N}) \in \R^{2N}\,, \qquad 
v_\ell =
\begin{cases} 
1 & \mbox{if } y_\ell< x\\
0 &  \mbox{if } y_\ell> x\,.
\end{cases} 
\end{equation}
By \eqref{eq:ssigma1} we obtain
\begin{equation}\label{eq:sigma-dot-vv}
\ssigma (t^n+) \cdot \vv = (-1)^n \frac{1}{2N} (\ssigma(0+)\cdot v_-) (v_- \cdot \vv)  + \BB_n\widetilde{\ssigma}(0+)\cdot \vv\,.
\end{equation}
Recalling the definition of \eqref{v-pm}, observe that $v_- \cdot \vv\in \{\pm1,0\}$ and hence 
\begin{align*}
|J(x,t^n+)| &= |\ssigma (t^n+) \cdot \vv| \\
&\le \frac{1}{2N} |\ssigma(0+)\cdot v_-|   +  |\BB_n\widetilde{\ssigma}(0+) \cdot \vv |\\
&\le \frac{1}{2N}  \tv \bar J_0  +  \|\BB_n\widetilde{\ssigma}(0+)\|_{\ell_1} 
\end{align*}
where \eqref{bound-on-comp-sigma-v} is used and an $\ell_1-\ell_\infty$ estimate is used for $\BB_n\widetilde{\ssigma}(0+) \cdot \vv$.

To complete the proof of \eqref{bound-on-sup-J}, it remains to bound the values of $J$ at  times $t\in (t^n+\DT/2, t^{n+1})$, since it may change due 
to the linear interaction of the waves. Recalling \eqref{B1}, we have
$$
\ssigma (t^{n+1}-) = B_1 \ssigma (t^n+) = (-1)^n \frac{1}{2N} (\ssigma(0+)\cdot v_-) B_1 v_-  + B_1 \BB_n\widetilde{\ssigma}(0+)
$$
with $B_1 v_- = -v_-$. By proceeding as before, we obtain 
\begin{align*}
|J(x,t^{n+1}-)| = | \ssigma(t^{n+1}-)\cdot \vv| 
& \le \frac{1}{2N} \tv \bar J_0  +  \|B_1\BB_n\widetilde{\ssigma}(0+)\|_{\ell_1}\\
&\le \frac{1}{2N}  \tv \bar J_0  +  \|\BB_n\widetilde{\ssigma}(0+)\|_{\ell_1}\,,
\end{align*}
where it is used that multiplication by $B_1$ leaves unaltered the $\ell_1$ norm (being a permutation matrix).
Therefore, \eqref{bound-on-sup-J}  is completely proved. 

\smallskip
\textbf{Proof of \eqref{bound-on-sup-rho}.}\quad 
For $x\not=x_j=j\DX$ and $x\not = y_\ell$, we have
\begin{align}\nonumber
\rho(x,t^n+) &= \rho(0+,t^n+)   + \sum_{y_\ell < x} \Delta \rho(y_\ell,t^n+) + \sum_{x_j < x} \Delta \rho(x_j,t^n+)
\end{align}
Recalling \eqref{int-rho-DX}, we have
$$
\left|\int_0^1 \rho(x,t^n+)\,dx \right| = \left|\int_0^1 \rho(x,0)\,dx\right| \le \DX \tv \rho_0\,,
$$
then
\begin{align*}
|\rho(0+,t^n+) | &\le \left| \int_0^1 [\rho(0+,t^n+) - \rho(x,t^n+)] \,dx\right| ~+~ \DX \tv \rho_0\\
& \le \sup_x \left|  \sum_{y_\ell < x} \Delta \rho(y_\ell,t^n+)\right| + \sup_x \left|  \sum_{x_j < x} \Delta \rho(x_j,t^n+)   \right|  ~+~ \DX \tv \rho_0
\end{align*}
and hence
\begin{align*}
|\rho(x,t^n+) | & \le 2 \sup_x   \underbrace{\left|\sum_{y_\ell < x} \Delta \rho(y_\ell,t^n+) \right|}_{(A)} + 2 \sup_x 
\underbrace{\left| \sum_{x_j < x} \Delta \rho(x_j,t^n+)\right|}_{(B)}   
~+~ \DX \tv \rho_0\,.
\end{align*}

\paragraph{$\bullet$ Estimate on (A).}\quad Recalling that $\Delta \rho(y_\ell) = \pm \sigma_{\pm1}$, we proceed similarly to 
\eqref{eq:sigma-dot-vv}:
\begin{align*}
\sum_{y_\ell < x} (\pm \sigma_{\pm1}) & = \ssigma(t^n+)\cdot \widetilde \vv \\
& = (-1)^n \frac{1}{2N} (\ssigma(0+)\cdot v_-) (v_- \cdot  \widetilde\vv)  + \BB_n\widetilde{\ssigma}(0+)\cdot \widetilde \vv
\end{align*}
where $\widetilde \vv = (v_1,\ldots,v_{2N}) \in \R^{2N}$, 
$$
v_\ell =
\begin{cases} 
1 & \mbox{if } y_\ell< x  \mbox{ and }\ell \mbox{ odd} \\
-1 & \mbox{if } y_\ell< x  \mbox{ and }\ell \mbox{ even} \\
0 &  \mbox{if } y_\ell> x\,.
\end{cases} 
$$
Hence $|v_- \cdot  \widetilde\vv|\le 2$ and then, by using \eqref{bound-on-comp-sigma-v}, we get:
\begin{align*}
|(A)| = |\sum_{y_\ell < x} (\pm \sigma_{\pm1})| &\le \frac{1}{N} |\ssigma(0+)\cdot v_-|  +  \|\BB_n\widetilde{\ssigma}(0+)\|_{\ell_1}\\
& \le \frac{1}{N} \tv \bar J_0   +  \|\BB_n\widetilde{\ssigma}(0+)\|_{\ell_1} \,.
\end{align*}

\paragraph{$\bullet$ Estimate on (B).}\quad Recalling that $\Delta \rho(x_j) = -2 g(J(x_j))\delta_j$, we have
\begin{align*}
(B) = 2 \left| \sum_{x_j < x} g(J(x_j,t^n+)) \delta_j \right| 
&\le 2 C_1  \max_j |J(x_j,t^n+)| \cdot \left(\sum_{j=1}^{N-1}  \delta_j \right)\\
&\le 2 C_1 \|k\|_{L^1}  \left(     \frac{1}{2N}  \tv \bar J_0  +  \|\BB_n\widetilde{\ssigma}(0+)\|_{\ell_1}  \right)\,.
\end{align*}
In conclusion, for every $x\in(0,1)$ we find that
\begin{align*}
\left|\rho(x,t^n+) \right| \le &~ 2 \DX \left(1 + C_1 \|k\|_{L^1}\right)  \tv \bar J_0  \\
& ~+ 2 \left(1 + 2 C_1 \|k\|_{L^1}\right)  \|\BB_n\widetilde{\ssigma}(0+)\|_{\ell_1} ~+~ \DX \tv \rho_0
\end{align*}
which is \eqref{bound-on-sup-rho} for $t\in (t^n, t^n+\DT/2)$. The estimate for $t\in (t^n+\DT/2,t^{n+1})$ is done similarly as the one for $J$.
\end{proof}

\begin{remark}\label{rem:TV-no-decay} (\emph{On the total variation of $J$}).
We remark that the total variation of $J_\DX$, being
$$
\tv J_\DX(\cdot,t)=\|\ssigma(t)\|_{\ell_1}\,,
$$
does not necessarily vanish at $t\to\infty$. Indeed, from \eqref{eq:ssigma1} it follows that
\begin{align*}
\|\ssigma(t^n+)\|_{\ell_1} 
&\ge \frac1{2N}  |\ssigma(0+)\cdot v_- |\left\|v_-\right\|_{\ell_1}-\left\|\BB_n  \widetilde{\ssigma}(0+)\right\|_{\ell_1}\\
&=  |\ssigma(0+)\cdot v_- | -\left\| \BB_n  \widetilde{\ssigma}(0+)\right\|_{\ell_1}
\end{align*}
where it is used that $\left\|v_-\right\|_{\ell_1}=2N$ (see the definition of $v_-$ at \eqref{v-pm}). 
By means of \eqref{eq:ssigma0v-}, and using the notation
$$
J_\ell=J(x_{\ell-1}+,0)=J(x_{\ell}-,0)=J_0(x_{\ell-1}+) \qquad \qquad \ell=1,\dots,N
$$
we have
\begin{align*}
\left|\ssigma(0+)\cdot v_-\right| &=\left|J_1-J_N +\sum_{\ell=1}^{N-1}(-1)^{\ell}(J_{\ell+1}-J_\ell)\right| \\
 &=2\left|J_1-J_N + \sum_{\ell=2}^{N-1}(-1)^{\ell-1}J_\ell\right|  =2\left|\sum_{\ell=1}^{N/2}(J_{2\ell-1}-J_{2\ell})\right|.
\end{align*}
If the initial datum $J_0(x)$ is strictly monotone, then
$$
\left|\ssigma(0+)\cdot v_-\right| = 2 \left|J_N - J_1\right|~ \to ~ 2   \left|J_0(1-) - J_0(0+) \right| = 2 \tv J_0 >0\,,  \quad N\to\infty\,.
$$
About the second term in the sum, when $\cc$ is constant in time we have $\BB_n = B(\cc)^n$ and
$$
\| B^n  \widetilde{\ssigma}(0+)\|_{\ell_1} \to 0 \qquad \text{as $n\to +\infty$}
$$
since $\widetilde{\ssigma}(0+)$ belongs to the subspace $E_- =<e,v_->^\perp$ corresponding to the eigenvalues with modulus $<1$. 
Therefore $\tv J(\cdot,t)$ does not tend to zero as $t\to +\infty$ for $J_0$ strictly monotone, and the limit is uniformly positive as $\DX=1/N\to 0$.

However, in \eqref{L-infty-decay}, it will turn out that the $L^\infty$-norm of $J$ is of order $\DX$ for large $t$.
\end{remark}

\subsection{A refined decomposition of the strength vector}

In this subsection we focus on the analysis of  $\|\BB_n\widetilde{\ssigma}(0+)\|_{\ell_1}$. In particular we analyze the sequence 
$\{\BB_n\widetilde{\ssigma}\}_{n\in\N}$ whenever $\widetilde{\ssigma}$ belongs to the subspace $E_- =<e,v_->^\perp$.

\smallskip
Let $N\in 2\N$ and consider $\widetilde{\ssigma}\in E_-$. By definition \eqref{v-pm} of $e$, $v_-$
then $\widetilde{\ssigma}$ satisfies
$$
\begin{cases}
\widetilde{\sigma}_1+\widetilde{\sigma}_2+\dots+\widetilde{\sigma}_{2N}=0,\\
\widetilde{\sigma}_1-\widetilde{\sigma}_2-\widetilde{\sigma}_3+\widetilde{\sigma}_4+\widetilde{\sigma}_5-\dots +\widetilde{\sigma}_{2N}=0,
\end{cases}
$$
which is equivalent to
$$
\begin{cases}
\widetilde{\sigma}_1+\widetilde{\sigma}_4+\dots +\widetilde{\sigma}_{2N-3}+\widetilde{\sigma}_{2N}=0,\\
\widetilde{\sigma}_2+\widetilde{\sigma}_3+\dots+\widetilde{\sigma}_{2N-2}+\widetilde{\sigma}_{2N-1}=0.
\end{cases}
$$
We introduce the following subspaces in $\R^{2N}$, each of dimension $N-1$: 
\begin{align*}
H_1\doteq \{(x_1,\dots,x_{2N})\in\R^{2N}&: \quad x_1+x_4+\dots+x_{2N-3}+x_{2N}=0\},\\
H_2\doteq \{(x_1,\dots,x_{2N})\in\R^{2N}&: \quad x_2+x_3+\dots+x_{2N-2}+x_{2N-1}=0\}\,.
\end{align*}
Hence we can write
\begin{equation}\label{def:decomp-ssigma} 
\widetilde{\ssigma}=\widetilde{\ssigma}' +\widetilde{\ssigma}'' \,, \qquad \widetilde{\ssigma}'\in H_1\,,\quad \widetilde{\ssigma}''\in H_2\,.
\end{equation}
Notice that, since $H_1$ and $H_2$ are complementary, we have
\begin{equation}\label{sum-ell1}
\left\|\widetilde{\ssigma}\right\|_{\ell_1}= 
\left\|\widetilde{\ssigma}'+\widetilde{\ssigma}''\right\|_{\ell_1}=
\left\|\widetilde{\ssigma}'\right\|_{\ell^1}+\left\|\widetilde{\ssigma}''\right\|_{\ell_1}\,.
\end{equation}
For later use, we define the following sets of indices
\begin{equation}\label{def:subsets}
\mathcal{I}'\doteq\{1,4,5,8,\dots,2N-3,2N\}\,,\qquad 
\mathcal{I}''\doteq \{2,3,6,7\dots,2N-2,2N-1\}\,.
\end{equation}

Let us define the vectors $\vvv_{ij}\in \R^{2N}$ for $i,j$ either $\in \mathcal{I}'$ or $\in \mathcal{I}''$ as follows,
\begin{equation}\label{eq:vij}
(\vvv_{ij})_i=1 \qquad \qquad (\vvv_{ij})_j=-1 \qquad \qquad (\vvv_{ij})_k=0 \quad \forall\, k\ne i,j\,.
\end{equation}
Remark that $\widetilde{\ssigma}'$ and $\widetilde{\ssigma}''$ can be written as a linear combination of suitable $\vvv_{ij}$'s, 
i.e.\ we can identify $\beta_{ij}',\beta_{ij}''\in \R$ such that
\begin{equation}\label{sigma-decomp}
\widetilde{\ssigma}'=\sum_{i,j\in \mathcal{I}'}\beta_{ij}'\vvv_{ij}\,, \qquad \widetilde{\ssigma}''=\sum_{i,j\in \mathcal{I}''}\beta_{ij}''\vvv_{ij}\,.
\end{equation}
By the triangular inequality, one has that
\begin{equation*}
\left\|\widetilde{\ssigma}'\right\|_{\ell_1}\le \sum_{ij}|\beta_{ij}'|  \left\|\vvv_{ij}\right\|_{\ell_1} = 2 \sum_{ij}|\beta_{ij}'|\,,
\qquad \left\|\widetilde{\ssigma}''\right\|_{\ell_1}\le  2 \sum_{ij}|\beta_{ij}''|\,.
\end{equation*}
In the next Lemma we prove that, for a suitable choice of the decomposition, the sum above can be made an equality.

\begin{lemma}\label{lem:decomp}
	\begin{itemize}
		\item[(i)] There exists a choice of the vectors $\vvv_{ij}$ such that \eqref{sigma-decomp} holds together with
\begin{align}
{\left\|\widetilde{\ssigma}'\right\|_{\ell_1}} & =2  \sum_{ij}|\beta_{ij}'|,\label{eq:betaij'}\\[2pt]
 {\left\|\widetilde{\ssigma}''\right\|_{\ell_1}}&=2 \sum_{ij}|\beta_{ij}''|.\label{eq:betaij''}
\end{align}

	\item[(ii)] The following estimate holds,
	\begin{equation}\label{eq:decomp-final}
\bigl\|\BB_n\widetilde{\ssigma}\bigr\|_{\ell_1}
\le \sup_{i,j}  \left\| \BB_n \frac{\vvv_{ij}}{\|\vvv_{ij}\|_{\ell_1}}\right\|_{\ell_1}   \cdot \bigl\|\widetilde\ssigma\bigr\|_{\ell_1},\qquad\quad \forall \widetilde\ssigma \in {E_-}.
	\end{equation}
	
\end{itemize}
\end{lemma}

\begin{proof}
We start with (i), it suffices to prove \eqref{eq:betaij'}, since \eqref{eq:betaij''} is analogous.

First, we have to find a suitable linear decomposition of $\widetilde{\ssigma}'(0+)$ in a basis of vectors of the form $\vvv_{ij}$, 
with $i,j\in \mathcal{I}'$. By construction we have
$$
\widetilde{\ssigma}'=\left(\widetilde{\sigma}'_1,0,0,\widetilde{\sigma}'_4,\widetilde{\sigma}'_5,0\dots,0,
\widetilde{\sigma}'_{2N-3},0,0,\widetilde{\sigma}'_{2N}\right),
$$
i.e.\ the components corresponding to indices in $\mathcal{I}''$ are zero. 
Therefore, we can simplify the notation and in place of $\widetilde{\ssigma}'$ consider 
$$
\mathbf{x}=\left(x_1,x_2,\dots,x_N\right)=\left(\widetilde{\sigma}'_1,\widetilde{\sigma}'_4,\dots,\widetilde{\sigma}'_{2N}\right)\in\R^N,
$$
the vector obtained erasing from $\widetilde{\ssigma}'$ the zero components and satisfying $x_1+x_2+\dots+x_N=0$.
Below we describe an algorithm to decompose $\textbf{x}$ along a basis of $\vvv_{ij}$'s, for $i,j\in\mathcal{I}'$. 

\medskip
\textbf{Step 1}. Let $\mathbf{x}\not=0$. Hence there exists a pair of indices $k_1,h_1\in\{1,\dots,N\}$ such that
$$
x_{k_1}\cdot x_{h_1}<0\,,\qquad 0<|x_{k_1}|=\min _{k=1,\dots,N;\, x_k\ne 0}|x_k|\,.
$$
In particular one has that $|x_{h_1}|\ge |x_{k_1}|$.

\medskip
\textbf{Step 2}.
Define the vector
$$
\textbf{x}^{(1)}\doteq \textbf{x}- x_{k_1} \vvv_{k_1 h_1} \in \R^N,
$$
and notice that it satisfies 
$$ 
\bigl(\textbf{x}^{(1)}\bigr)_{k} =
\begin{cases} 
0&  k=k_1\\
x_{h_1}+x_{k_1}  & k=h_1\\
x_{k} & k\not = k_1,\, h_1\,.
\end{cases}
$$
In particular, 
$$
 \bigl|\bigl(\textbf{x}^{(1)}\bigr)_{h_1}\bigr|=|x_{h_1}|-|x_{k_1}| \ge 0 
$$
and hence 
$$
\bigl\| \textbf{x}^{(1)}\bigr\|_{\ell_1}=\bigl\|\textbf{x}\bigr\|_{\ell_1}-2|x_{k_1}|<\bigl\|\textbf{x}\bigr\|_{\ell_1}.
$$

\medskip
\textbf{Step 3}.
We apply the same procedure to $\textbf{x}^{(1)}$, namely we choose suitable indexes $k_2,h_2\in\{1,\dots,N\}$ such that 
\begin{gather*}
\bigl(\textbf{x}^{(1)}\bigr)_{k_2}\cdot \bigl(\textbf{x}^{(1)}\bigr)_{h_2} <0\,,\qquad
0<\bigl|\bigl(\textbf{x}^{(1)}\bigr)_{k_2}\bigr|= \min_{k=1,\dots,N
\,, \bigl(\textbf{x}^{(1)}\bigr)_{k}\ne 0}|\bigl(\textbf{x}^{(1)}\bigr)_{k}|\,.
\end{gather*}
Notice that, since $\bigl(\textbf{x}^{(1)}\bigr)_{k_1}=0$, one has that $k_2$, $h_2$ are different from $k_1$.
Moreover one has
$\bigl|\bigl(\textbf{x}^{(1)}\bigr)_{h_2}\bigr|\ge \bigl|\bigl(\textbf{x}^{(1)}\bigr)_{k_2}\bigr|$. 

\smallskip
As in Step \textbf{2}, we define
\begin{align*}
\textbf{x}^{(2)}&\doteq \textbf{x}^{(1)}- \bigl(\textbf{x}^{(1)}\bigr)_{k_2} \vvv_{k_2 h_2}\\
& = \textbf{x}- x_{k_1} \vvv_{k_1 h_1}-  \bigl(\textbf{x}^{(1)}\bigr)_{k_2}\vvv_{k_2 h_2},
\end{align*}
that is
$$ 
\bigl(\textbf{x}^{(2)}\bigr)_{k}
=
\begin{cases} 
0&  k=k_2\\
\bigl(\textbf{x}^{(1)}\bigr)_{h_2}+ \bigl(\textbf{x}^{(1)}\bigr)_{k_2}
& k=h_2\\
\bigl(\textbf{x}^{(1)}\bigr)_{k} & k\not = k_2,\, h_2\,.
\end{cases}
$$
Notice that
$$
\bigl(\textbf{x}^{(2)}\bigr)_{k} =0 \qquad \mbox{ for } k=k_1,\ k_2
$$
and that
$$
 \bigl|\bigl(\textbf{x}^{(2)}\bigr)_{h_2}\bigr|=|\bigl(\textbf{x}^{(1)}\bigr)_{h_2}|-|\bigl(\textbf{x}^{(1)}\bigr)_{k_2} |  \ge 0 \,.
$$
Observe that $|x_{k_1}|+\bigl| \bigl(\textbf{x}^{(1)}\bigr)_{k_2}\bigr| \le |x_{k_1}|+|x_{k_2}|$ and
\begin{align*}
\bigl\| \textbf{x}^{(2)}\bigr\|_{\ell_1} &= \bigl\| \textbf{x}^{(1)}\bigr\|_{\ell_1} - 2 |\bigl(\textbf{x}^{(1)}\bigr)_{k_2} | \\
&=\bigl\|\textbf{x}\bigr\|_{\ell_1}-2 \left(|x_{k_1}|+\bigl| \bigl(\textbf{x}^{(1)}\bigr)_{k_2}\bigr|\right)\,. 
\end{align*}

\medskip
\textbf{Step 4}.
Proceeding by induction, after at most $N-1$ iterations of the method we get
$$
\textbf{x}^{(N-1)}\doteq \textbf{x}- x_{k_1} \vvv_{k_1 h_1} - \bigl(\textbf{x}^{(1)}\bigr)_{k_2} \vvv_{k_2 h_2} - \dots - 
\bigl(\textbf{x}^{(N-2)}\bigr)_{k_{N-1}} \vvv_{k_{N-1} h_{N-1}}= (0,\dots,0)\in\R^N.
$$
Thus, 
\begin{equation}\label{eq:texbfx}
0=\bigl\|\textbf{x}^{(N-1)}\bigr\|_{\ell_1}=\bigl\|\textbf{x}\bigr\|_{\ell_1}- 2\left(\sum_{i=1}^{N-1}|x_{k_i}|\right).
\end{equation}
and hence
$$
\bigl\|\textbf{x}\bigr\|_{\ell_1} = \frac 12 \, {\sum_{i=1}^{N-1}|x_{k_i}|}\,.
$$
Since we can write that $\sum_{i=1}^{N-1}|x_{k_i}|= \sum_{ij}|\beta_{ij}'|$,
then the proof of \eqref{eq:betaij'} is complete.

 \paragraph{Proof of (ii).} By using \eqref{def:decomp-ssigma}, we have
$$
\BB_n \widetilde{\ssigma} = \BB_n \widetilde{\ssigma}' + \BB_n\widetilde{\ssigma}''\,.
$$
By means of  \eqref{sigma-decomp} and (i) we find that
\begin{align*}
\bigl\|  \BB_n\widetilde{\ssigma}\bigr\|_{\ell_1}&
\le  \bigl\|  \BB_n\widetilde{\ssigma}'\bigr\|_{\ell_1} 
+ \bigl\| \BB_n  \widetilde{\ssigma}''\bigr\|_{\ell_1}\\
&\le   \sum_{\mathcal{I}'}|\beta_{ij}'| \bigl\|  \BB_n \vvv_{ij}\bigr\|_{\ell_1}
+ \sum_{\mathcal{I}''}|\beta_{ij}''|\bigl\| \BB_n \vvv_{ij}\bigr\|_{\ell_1}
\\
&\le \left(\sum_{\mathcal{I}'}|\beta_{ij}'| + \sum_{\mathcal{I}''}|\beta_{ij}''|\right) 
\sup_{i,j} \bigl\| \BB_n \vvv_{ij}\bigr\|_{\ell_1} 
\\
&\le\frac 12  \left({\bigl\|\widetilde{\ssigma}'\bigr\|_{\ell_1}}+ {\bigl\|\widetilde{\ssigma}''\bigr\|_{\ell_1}}\right)
\sup_{i,j} \bigl\| \BB_n \vvv_{ij}\bigr\|_{\ell_1}\,. 
\end{align*}
As $\|\vvv_{ij}\|_{\ell_1} =2$ and by using \eqref{sum-ell1}, the proof of \eqref{eq:decomp-final} is complete.
\end{proof}
Thanks to Lemma~\eqref{lem:decomp}, especially \eqref{eq:decomp-final}, it is then sufficient to study the behaviour of 
$$
\sup_{i,j}  \left\| \BB_n \frac{\vvv_{ij}}{\|\vvv_{ij}\|_{\ell_1}}\right\|_{\ell_1}
$$
as $n\to\infty$ for every $\vvv_{ij}$, as defined in \eqref{eq:vij}, with either $i,j\in \mathcal{I}'$ or $i,j\in \mathcal{I}''$.
The goal is to prove that the above quantity decays exponentially fast as $n\to\infty$, uniformly for large $N$.



\subsection{Linear damping}\label{Subsect_5.2}
In this subsection we consider the special case when $\cc$ is constant in space and time (which is the case if $k$ is constant and $g$ is linear) 
and hence $B(\cc)$ does not depend on time. This means that the product of the matrices in \eqref{def:BB} reduces to the 
$n^{th}$ power of $B(\cc)$. In particular we focus on the structure of the power for $n=2N,$ since we can exploit the fact that the permutation
$B(0)^{2N}$ is the identity. 

We remark that all the quantities in this subsection do not depend on the initial data; they depend only on the coefficients of the 
system~\eqref{DWE-rho-J-IBVP}.

\smallskip
Assume that 
$$
k(x) = \bar k>0 \quad \forall\, x\in(0,1)\,,\qquad g'(J) = const. =  C_1 
$$
and set
\begin{equation}\label{eq:gamma=dN}
d  ~\dot= ~ \bar k\, C_1 \,,\qquad  \gamma ~\dot =~ \frac{d}{N}\,.
\end{equation}
By Proposition~\ref{prop:4.3} and Birkhoff Theorem, the matrix $B(\cc)$ can be written as
\begin{align*}
B(\cc) =(1-c) B(0) + c B_1  = (1-c) \left[ B(0) + \frac{c}{1-c} B_1\right] \,,
\end{align*}
where $\cc= c(1,1,\dots,1)\in \R^{N-1}$ and 
$$
c=  \frac{\gamma}{\gamma+1} 
\,,  \qquad   \frac{c}{1-c} = \gamma = \frac {d}N\,. 
$$
Hence 
\begin{align} 
B(\cc)^{2N} & =(1-c)^{2N}\left[B(0)+ \gamma B_1 \right]^{2N}  \label{eq:expansion}
\end{align}
It is clear that 
\begin{equation*}
(1-c)^{2N} = \left(1+ \frac {d}N \right)^{-2N} \to \ee^{-2d}\,,\quad N\to\infty\,.
\end{equation*}
Let us focus on the second factor in \eqref{eq:expansion}, that is
\begin{equation}\label{eq:expansion-2}
\left[B(0)+ \gamma B_1 \right]^{2N}  = \sum_{k=0}^{2N} \gamma^k  S_k(B(0),B_1), 
\end{equation}
where each term $S_k(B(0),B_1)$ is the sum of all products of $2N$ matrices which are either $B_1$ or $B(0)$,
and in which $B_1$ appears exactly $k$ times, that is
\begin{equation}\label{def:S_k}  
\left\{
\begin{aligned}
S_k(B(0),B_1)= & \sum_{(\ell_1,\ldots,\ell_{k+1})} 
B(0)^{\ell_1} \cdot B_1 \cdot B(0)^{\ell_2} \cdot B_1 \cdots B(0)^{\ell_k} \cdot B_1 \cdot B(0)^{\ell_{k+1}}\\
&0\le \ell_j\le2N-k\,,\qquad   \sum_{j=1}^{k+1} \ell_j = 2N-k\,.
\end{aligned}\right.
\end{equation}
In what follows we use extensively the fact that
$B_1^2=I_{2N}=B(0)^{2N}$ and the commutation property described in next proposition. 

\begin{proposition} 
The following identity holds for any $\ell\in\N$:
\begin{equation}\label{commut-rule}
B(0)^{\pm\ell} B_1 = B_1  B(0)^{\mp\ell}.
\end{equation}
\end{proposition}
\begin{proof} Recalling \eqref{B1}--\eqref{B-of-cc}, we have that $B(0)^{-1}= (B_2(0) B_1)^{-1} = B_1 B_2(0)$. 
Then for every $\ell\ge 0$ we have
\begin{align*}
B(0)^{-\ell} B_1 &= \underbrace{(B_1 B_2(0))\cdots (B_1 B_2(0))}_{\ell\ \mbox{times}}\cdot B_1 \\
& = B_1\cdot  \underbrace{(B_2(0) B_1)\cdots (B_2(0) B_1)}_{\ell\ \mbox{times}}\\
&= B_1  \cdot B(0)^{\ell}\,.
\end{align*}
As for the identity for $+\ell$, notice that
\begin{align*}
B(0)^{\ell} B_1 &= B(0)^{2N - (2N-\ell)} B_1 = B(0)^{2N} B(0)^{-(2N-\ell)} B_1 \\
&= B(0)^{-(2N-\ell)} B_1,
\end{align*}
where we used that $B(0)^{2N} = I_{2N}$. 
Hence, by the first identity we get 
$$
B(0)^{\ell} B_1 = B_1  \cdot B(0)^{2N-\ell} = B_1  \cdot B(0)^{-\ell}\,.   
$$
\end{proof}
By means of \eqref{commut-rule} and using that $B_1^2=I_{2N}$, the generic term in the sum $S_k$ in \eqref{def:S_k} 
can be conveniently rewritten. Indeed, one has $S_0 = S_{2N} = I_{2N}$. 
For $k=1,\dots, 2N-1$, we have to distinguish the case of even/odd $k$. 

\medskip
$\bullet$ For $k$ even, we have
\begin{equation}\label{eq:gen-term-S_k-even}
B(0)^{\ell_1} \cdot B_1 \cdot B(0)^{\ell_2} \cdot B_1 \cdots B(0)^{\ell_k} \cdot B_1 \cdot B(0)^{\ell_{k+1}}= B(0)^{\alpha-\beta}\,,
\end{equation}
where
\begin{equation}\label{def:alpha-beta}
\alpha=\sum_{j=1,\ j \mbox{ {\small  odd}}}^{k+1}\ell_j\,,\qquad  \beta =\sum_{j=2,\ j \mbox { {\small even}}}^{k+1}\ell_j = 2N-k -\alpha\,.
\end{equation}
Now let us count how many vectors $(\ell_1,\ldots,\ell_{k+1})$ lead, thanks to \eqref{eq:gen-term-S_k-even}, to the same matrix 
$$B(0)^{\alpha-\beta} =  B(0)^{2 \alpha+k}\,.
$$
In the first sum of \eqref{def:alpha-beta} the indices are $k/2 +1$, while in the second sum 
they are $k/2$. Hence, for a given $\alpha$, the number of the distinct vectors $(\ell_1,\ldots,\ell_{k+1})$ for which \eqref{def:alpha-beta} holds is
\footnote{Given $M\ge0$ and $a_j\ge0$ integers such that $\sum_{j=1}^n a_j=M$, the number of distinct $(a_1,\ldots,a_n)$ is equal to the binomial coefficient $
\begin{pmatrix} M+n-1 \\ n-1 \end{pmatrix}= \begin{pmatrix} M+n-1 \\ M \end{pmatrix}\,.
 $  }
\begin{equation}\nonumber
\begin{pmatrix} \alpha + \frac k2\\ \frac k2      \end{pmatrix} \begin{pmatrix} 2N-\alpha-1- \frac k2\\ \frac k2 -1      \end{pmatrix}\,,\qquad
\alpha=0,\ldots,2N-k\,.
\end{equation}
If we perform a change of variable $j=\alpha+ k/2$, we get
\begin{equation*}
\begin{pmatrix} j \\ \frac k2      \end{pmatrix} \begin{pmatrix} 2N-j-1\\ \frac k2 -1      \end{pmatrix}
\,,\qquad
j=\frac k2,\ldots,2N-\frac k2\,,
\end{equation*}
and 
\begin{equation}\label{eq:Skeven}
\boxed{S_k(B(0),B_1)
=  \sum_{j=\frac k2 }^{2N-\frac k2}  
\begin{pmatrix} j \\ \frac k2      \end{pmatrix} \begin{pmatrix} 2N-j-1\\ \frac k2 -1      \end{pmatrix}
B(0)^{2j}\,,\qquad k=2,4,\dots, 2N\,. }
\end{equation}

\medskip
$\bullet$ For $k$ odd, we have
\begin{align*}
B(0)^{\ell_1} \cdot B_1 \cdot B(0)^{\ell_2} \cdot B_1 \cdots B(0)^{\ell_k} \cdot B_1 \cdot B(0)^{\ell_{k+1}}
&= B(0)^{\alpha-\beta} B_1\\ 
&=B(0)^{2 \alpha+k} B_1 \\
& =B(0)^{2 \alpha+k-1} B_2(0)\,,
\end{align*}
where $\alpha$, $\beta= 2N - k - \alpha$ are given in \eqref{def:alpha-beta}.

Here, the number of vectors $(\ell_1,\ldots,\ell_{k+1})$ for which \eqref{def:alpha-beta} holds are counted as follows. The indices $\ell_j$ 
are in total $(k+1)/2$ for both sums, hence for a given $\alpha$ the number of terms is
\begin{equation}\nonumber
\begin{pmatrix} \alpha + \frac {k-1}2\\ \frac{k-1}2       \end{pmatrix} \begin{pmatrix} 2N-\alpha- \frac{k-1}2 -1\\  \frac{k-1}2      \end{pmatrix}
\,,\qquad  \alpha=0,\ldots,2N-k\,.
\end{equation}
If we perform a change of variable $j=\alpha+ \frac {k-1}2$, we get
\begin{equation*}
\begin{pmatrix} j \\ \frac{k-1}2       \end{pmatrix} \begin{pmatrix} 2N-j -1 \\  \frac{k-1}2      \end{pmatrix}\,,\qquad
j=\frac {k-1}2,\ldots,2N-\frac {k+1}2\,.
\end{equation*}
Hence, 
\begin{equation}\label{eq:Skodd}
\boxed{
S_k(B(0),B_1) =  \sum_{j=\frac{k-1}2 }^{2N-\frac{k+1}2}   \begin{pmatrix} j \\ \frac{k-1}2     \end{pmatrix} \begin{pmatrix} 2N-j-1\\ \frac{k-1}2     
\end{pmatrix} B(0)^{2j}  B_2(0)\qquad k=1,3,\dots, 2N-1\,.}
\end{equation}

\smallskip
The next proposition gives an explicit formula for the sum of the powers of $B(0)$.

\begin{proposition}
Let $\widehat P$ be the matrix defined by
\begin{equation}\label{def:hat_P}
\widehat P  ~\dot =~  \frac 12 \left(e e^t + v_- v_-^t \right)\,,
\end{equation}
which is the matrix composed by $N^2/4$ squared blocks as 
$$
\begin{bmatrix}
1&  0  & 0 & 1\\
0 &1 &1 &0 \\
0& 1& 1 & 0  \\
1&  0  & 0 & 1
\end{bmatrix}.
$$
Then, the following identity holds:
\begin{equation}\label{eq:sum-full-cycle}
 \sum_{j=0}^{N-1} B(0)^{2j} =  \sum_{j=1}^{N} B(0)^{2j} =\widehat P    \,.
\end{equation}
\end{proposition}
\begin{proof} The first equality in \eqref{eq:sum-full-cycle} follows from the following identity: 
$$
\left(I_{2N} - B(0)^2 \right)  \left( \sum_{j=0}^{N-1} B(0)^{2j}\right)= 0\,.
$$
Indeed, 
$$\left(I_{2N} - B(0)^2 \right) \left(  \sum_{j=0}^{N-1} B(0)^{2j}\right) = \left( \sum_{j=0}^{N-1} B(0)^{2j}\right) - \left( \sum_{j=1}^{N} B(0)^{2j}\right)
= I_{2N} - B(0)^{2N} = 0.
$$
To prove the second identity in \eqref{eq:sum-full-cycle}, observe that the matrix $B(0)^{2}$ contains the following two separated "cycles" of length $N$,
\begin{align*}
&1\to 5 \to 9\to \ldots \to 2N-3 \to 2N \to 2N-4 \to \ldots \to 4 \to 1\\
&2\to 3 \to 7\to \ldots \to 2N-1 \to 2N-2 \to  2N-6 \to \ldots \to 6 \to 2\,.
\end{align*}
In the first, second case the indexes are exactly the ones in $\mathcal{I}'$, $\mathcal{I}''$ respectively.

By summing all the permutations $B(0)^{2}$, \dots, $B(0)^{2N}= I_{2N}$ one obtains that every $i^{th}$ row, with $i\in \mathcal{I}'$, 
has value =1 exactly at every index $\in \mathcal{I}'$ and value =0 otherwise. The same holds for every 
$i^{th}$ row with $i\in \mathcal{I}''$\,. Hence \eqref{eq:sum-full-cycle} holds.
\end{proof}

The next theorem provides an estimate on the components of $B(\cc)^{2N}$ in terms of $d$, $N$.

\begin{theorem}\label{theo:main}
Let $N\in 2\N$\,. The following bound holds true:
\begin{align}\label{eq:theoB2N-bis}
\left[B(0)+ {\frac d N}  B_1 \right]^{2N} & =~  
I_{2N} +  {\frac {2d}N} \widehat P + \sum_{j=0}^{2N-1}   \zeta_{j,N} B(0)^{2j}B_2(0) + \sum_{j=1}^{2N-1}  \eta_{j,N} B(0)^{2j}\,,
\end{align}
where
\begin{align}\label{stima-su-zeta_jN}
0\le \sum_{j=0}^{2N-1} \zeta_{j,N} &\le \sinh(2d)  - 2d+ {\frac1N} f_0(d) 
\\
0\le \sum_{j=1}^{2N-1} \eta_{j,N} &\le  \cosh(2d)-1  ~+~  {\frac1N} f_1(d)\,, 
\label{stima-su-eta_jN}
\end{align}
and 
\begin{align}\label{modif-bessel-0}
f_0(d) &~\dot =~ \sum_{\ell=1}^{\infty} \frac{d^{2\ell +1}}{(\ell!)^2}  = d \left[ I_0(2d) -1 \right]\\
f_1(d) &~\dot =~ \sum_{h=1}^{\infty}\frac{d^{2h}}{h! (h-1)!} = d I_1(2d)  \label{modif-bessel-1}\,,
\end{align}
where 
$$
I_\alpha(2x) = \sum_{m=0}^\infty \frac{x^{2m+\alpha}}{m! (m+\alpha)!} \,,\qquad \alpha=0,1
$$
is a modified Bessel function of the first type, see \cite [p. 222]{Spec-fn}.
\end{theorem}

\begin{proof} 
From the identity \eqref{eq:expansion-2} we have
\begin{align}\label{eq:B2N}
 \left[B(0)+ \gamma B_1 \right]^{2N} =& I_{2N}
 + \left[\sum_{\substack{k=1\\k \, \text{odd}}}^{2N-1} + 
\sum_{\substack{k=2\\k\,\text{even}}}^{2N} \right] \gamma^k S_k(B(0),B_1) 
\end{align}

First, let us focus on the sum with $k$ odd in \eqref{eq:B2N}. By \eqref{eq:Skodd}, we substitute the expression for $S_k$ 
and exchange the sum in $k$ and $j$ to get
\begin{align}\label{eq:part1}
\sum_{\substack{k=1\\k \, \text{odd}}}^{2N-1}\gamma^k S_k(B(0),B_1)&=
\sum_{j=0}^{2N-1}  \widetilde \zeta_{j,N} B(0)^{2j}B_2(0),
\end{align}
where
\begin{align*}\nonumber
\widetilde \zeta_{j,N}&=\sum_{\substack{k=1\\k\,\text{odd}}}^{\min\{2j+1,4N-2j-1\}}\gamma^k \begin{pmatrix}
j \\ \frac{k-1}{2}
\end{pmatrix} \begin{pmatrix}
2N-j-1 \\ \frac{k-1}{2} 
\end{pmatrix}\\  \nonumber
&=\sum_{\substack{\ell=0}}^{\min\{j,2N-j-1\}}\gamma^{2\ell +1} \begin{pmatrix}
j \\ \ell
\end{pmatrix} \begin{pmatrix}
2N-j-1 \\ \ell
\end{pmatrix} \,.
\end{align*}  
It is convenient to separate, in the expression of $\widetilde \zeta_{j,N}$, the term with $\ell=0$ and the sum for $\ell\ge1$, 
since the former does not depend on $j,N$: 
\begin{equation}  \label{def:binomial-coeff}
\widetilde \zeta_{j,N} = \gamma + \zeta_{j,N} \,,\qquad    
\zeta_{j,N} \dot = \sum_{\ell=1}^{\min\{j,2N-j-1\}}\gamma^{2\ell +1} \begin{pmatrix} j \\ \ell
\end{pmatrix} \begin{pmatrix}
2N-j-1 \\ \ell
\end{pmatrix} \,.
\end{equation}
Next we provide an estimate on the coefficients $\zeta_{j,N}$. Using the inequality 
$$
\begin{pmatrix} n \\ k
\end{pmatrix} \le \frac{n^k}{k!}\,,\qquad 0\le k\le n
$$
and the definition $\gamma=d/N$, we find that
\begin{align}\label{ineq:zeta_jN}
\zeta_{j,N}&\le 
\sum_{\substack{\ell=1}}^{\infty}\frac{d^{2\ell+1}}{(\ell!)^2} \frac{ j^\ell}{N^\ell}  \frac{(2N-j-1)^\ell}{N^\ell} \,.
\end{align}  
Now we introduce another change of variable, 
\begin{equation}\label{def:x-j}
x_j = -1 + \frac j N\,,\qquad \frac{j}{N} = (1+x_j)\,,\qquad j=0,\ldots, 2N-1\,.
\end{equation}
Thanks to the inequality \eqref{ineq:zeta_jN} we get
\begin{align}\nonumber
0\le \zeta_{j,N}&\le \sum_{{\ell=1}}^{\infty}\frac{d^{2\ell+1}}{(\ell!)^2} {(1+x_j)}^\ell  \left(1-x_j- \frac 1 N\right)^\ell \\
&\le \sum_{{\ell=1}}^{\infty}\frac{d^{2\ell+1}}{(\ell!)^2} {(1-x_j ^ 2 )}^\ell \,.\nonumber 
\end{align}  
As a consequence, we deduce an estimate for the sum of the $\zeta_{j,N}$:
\begin{align*} 
0\le \sum_{j=0}^{2N-1} \zeta_{j,N} &\le {\frac 1 N} \sum_{j=0}^{2N-1} \sum_{{\ell=1}}^{\infty}\frac{d^{2\ell+1}}{(\ell!)^2} {(1-x_j ^ 2 )}^\ell \\
&=  \sum_{\ell=1}^{\infty} \frac{d^{2\ell+1}}{(\ell!)^2} \left\{\frac 1N \sum_{j=0}^{2N-1} {(1-x_j ^ 2 )}^\ell \right\}
\end{align*}
where we used that $\DX=1/N$. Using the definition \eqref{def:x-j} we notice that
\begin{equation*}
\frac 1 N\sum_{j=0}^{2N-1} {(1-x_j ^ 2 )}^\ell ~~\to~~ \int_{-1}^1 (1-x^2)^\ell\,dx \quad \mbox{as}\  {N\to\infty}, \qquad \ell\ge 1\,;
\end{equation*}
more precisely the following estimate holds,
\begin{align}\nonumber
\sum_{j=0}^{2N-1} {(1-x_j ^ 2 )}^\ell\DX & =  \left(\sum_{j=0}^{N-1} + \sum_{j=N+1}^{2N-1}\right) {(1-x_j ^ 2 )}^\ell\DX ~+~ \DX\\
&\le \int_{-1}^1 (1-x^2)^\ell\,dx ~+~ \DX\,. \label{int-bessel}
\end{align}
Since $(1+2\ell)! = (1+2\ell)!!  \cdot 2^\ell \cdot \ell ! $, it is easy to check the following identities
\begin{equation}\label{integrals}
\int_{-1}^1 (1-x^2)^\ell\,dx = \frac{2^{\ell+1} \cdot \ell!} {(1+2\ell)!!}  =    \frac{2^{2\ell+1}\cdot (\ell!)^2} {(1+2\ell)!} \qquad \ell\ge 1\,.  
\end{equation}
By plugging the previous estimates into the sum of the $\zeta_{j,N}$ we get
\begin{align*} 
0\le \sum_{j=0}^{2N-1} \zeta_{j,N} 
& \leq \sum_{\ell=1}^{\infty} \frac{d^{2\ell+1}}{(\ell!)^2}  \frac{2^{2\ell+1}\cdot (\ell!)^2} {(1+2\ell)!} ~+~  
\DX \underbrace{\sum_{\ell=1}^{\infty} \frac{d^{2\ell +1}}{(\ell!)^2}}_{\dot = f_0(d)}\\
&=\sum_{\ell=1}^{\infty} \frac{(2d)^{2\ell+1}} {(1+2\ell)!} ~+~  \DX f_0(d)\\
&=\sinh(2d)  -2d ~+~  \DX f_0(d)\,.
\end{align*}
Therefore \eqref{stima-su-zeta_jN} follows.

\smallskip
Analogously we treat the sum with $k$ even in \eqref{eq:B2N}. By \eqref{eq:Skeven} we can exchange the sum in $k$ and $j$,  
hence we rewrite this term as 
\begin{align}\label{eq:part2}
\sum_{\substack{k=2\\k\,\text{even}}}^{2N} \gamma^k S_k(B(0),B_1)=
\sum_{j=1}^{2N-1}  \eta_{j,N} B(0)^{2j},
\end{align} 
where we set
\begin{align*}
\eta_{j,N} ~&\dot = \sum_{\substack{k=2\\ k\, \text{even}}}^{\min\{2j,4N-2j\}}\gamma^k 
\begin{pmatrix}
j \\ \frac{k}{2}
\end{pmatrix} \begin{pmatrix}
2N-j -1 \\ \frac{k}{2}-1
\end{pmatrix}\\
&= \sum_{h=1}^{\min\{j,2N-j\}}\gamma^{2h} 
\begin{pmatrix}
j \\ h
\end{pmatrix} \begin{pmatrix}
2N-j -1 \\ h -1
\end{pmatrix}.
\end{align*}
Similarly to the estimate \eqref{ineq:zeta_jN} for $\zeta_{j,N}$ and using the change of variables \eqref{def:x-j}, we find that
\begin{align*}
\eta_{j,N} 
&\le   \frac 1N  \sum_{h=1}^{\infty}\frac{d^{2h}}{h! (h-1)! }  (1+x_j)^h \left(1-x_j - \frac 1N \right)^{h-1} \\
&\le  \frac 1N  \sum_{h=1}^{\infty}\frac{d^{2h}}{h! (h-1)!}  (1- x_j^2)^{h-1} \left(1+x_j  \right)\,.
\end{align*}
The sum of the $\eta_{j,N} $  can be estimated as follows,
\begin{align*}
\sum_{j=1}^{2N-1} \eta_{j,N} &\le   \sum_{h=1}^{\infty}\frac{d^{2h}}{h! (h-1)!} \left\{ \frac 1N \sum_{j=1}^{2N-1} (1- x_j^2)^{h-1} \left(1+x_j  \right)  \right\} \,.
\end{align*}
By definition of the \eqref{def:x-j} and simmetry we have
$$
\sum_{j=1}^{2N-1} (1- x_j^2)^{h-1} x_j =0\,,
$$
while by \eqref{int-bessel} with $\ell=h-1$ and by \eqref{integrals} we find that
\begin{align*}
\frac 1N  \sum_{j=0}^{2N-1} {(1-x_j ^ 2 )}^{h-1} &\le  \int_{-1}^1 (1-x^2)^{h-1}\,dx ~+~ \frac 1N \\
& =  \frac{2^{2h-1}\cdot ((h-1)!)^2} {(2h-1)!}  ~+~ \frac 1N\,.
\end{align*}
Therefore 
\begin{align*}
\sum_{j=1}^{2N-1} \eta_{j,N} &\le  \sum_{h=1}^{\infty}\frac{d^{2h}}{h! (h-1)!}   \frac{2^{2h-1}\cdot ((h-1)!)^2} {(2h-1)!}   ~+~\frac 1N
\underbrace{\sum_{h=1}^{\infty}\frac{d^{2h}}{h! (h-1)!}}_{\dot = f_1(d)}\\
&= \sum_{h=1}^{\infty}\frac{(2d)^{2h}}{2h (2h-1)!}   ~+~\frac{f_1(d)}N\\
&= \sum_{h=1}^{\infty}\frac{(2d)^{2h}}{(2h)!}   ~+~ \frac{f_1(d)}N \\
&= \cosh(2d)-1  ~+~\frac{f_1(d)}N\,,
\end{align*}
that leads to \eqref{stima-su-eta_jN}.
\end{proof}

\begin{remark} For $a\in\R$ and $n\ge0$, $n$ integer, we introduce the notation (\textit{shifted factorial}, see \cite[p. 2]{Spec-fn}):
\begin{equation}\label{def:rising-symbol}
(a)_n = \begin{cases}
1& n=0\\
a(a+1) \cdots (a+n-1) & n\ge 1\,.
\end{cases}
\end{equation}
With this notation we can write $(1)_n = n!$. Observe that, if $a$ is a negative integer, then $(a)_n$ vanishes for every $n\ge |a|+1$.
Then the product of the binomial coefficients in \eqref{def:binomial-coeff} can be rewritten as follows,
\begin{equation*}
\begin{pmatrix}
j \\ \ell
\end{pmatrix} \begin{pmatrix}
2N-j-1 \\ \ell
\end{pmatrix} = \frac1{(\ell!)^2}  
(-j)_\ell 
\cdot   (-2N+j+1)_\ell\,, \qquad  \ell\ge 0\,,
\end{equation*}
and it is clear that the above quantity vanishes for $\ell > \min\{j,2N-j-1\}$. Therefore the coefficients $\zeta_{j,N}$ is rewritten as
\begin{align}\label{eq:zeta_jN}
\zeta_{j,N}&=\sum_{\substack{\ell=1}}^{\infty}\frac{\gamma^{2\ell+1}}{\ell!} \frac{ (-j)_\ell  (-2N+j+1)_\ell} {(1)_\ell} \,.
\end{align}  
The coefficients $\widetilde \zeta_{j,N}$ in \eqref{eq:part1} can be rewritten in terms of the hypergeometric function, see \cite{Spec-fn},
\begin{equation*}
{}_2 F_1(a,b,c;z) = \sum_{n=0}^\infty \frac{(a)_n (b)_n}{(c)_n} \frac{z^n}{n!}\,,\qquad a,b,c\in\R\,.
\end{equation*}
In conclusion we have
\begin{align}\nonumber
\widetilde  \zeta_{j,N}&= \gamma   ~ {}_2 F_1(-j, -2N+j+1 , 1;   \gamma^2)\,,\qquad \gamma=\frac d N
\end{align}
and hence, from \eqref{eq:part1}, we obtain:
\begin{align}\nonumber
\sum_{\substack{k=1\\k \, \text{odd}}}^{2N-1}\gamma^k S_k
&= \gamma ~ \sum_{j=0}^{2N-1} {}_2 F_1(-j, -2N+j+1 , 1;   \gamma^2) B(0)^{2j}B_2(0)\,.
\end{align}
\end{remark}

Next, we want to prove a contractive estimate for $\|B(\cc)^{2N}\vvv_{ij}\|_{\ell_1}$. We recall that here $\cc = c(1,\ldots,1)\in \R^{N-1}$ with
$c=d/N$ for some $d>0$.

\begin{proposition}\label{prop:5.8}
	Let $i,j$ be indices both either $\in\mathcal{I}'$ or $\in\mathcal{I}''$ (see \eqref{def:subsets}). 
For every $d>0$ there is a constant $C_N(d)>0$ such that 
\begin{equation}\label{eq:norm1}
\bigl\|B(\cc)^{2N}\vvv_{ij} \bigr\|_{\ell_1} \le C_N(d) \bigl\|\vvv_{ij}\bigr\|_{\ell_1},
\end{equation}
where $\vvv_{ij}$ are defined at \eqref{eq:vij} and 
\begin{equation}\label{eq:Ctilde}
C_N(d)\to (1-2de^{-2d}) <1\,,\qquad N\to\infty \,.
\end{equation}
\end{proposition}

\begin{proof}
Notice that 
\begin{align*}
B(\cc)^{2N}\vvv_{ij}= B(\cc)^{2N} e_{i} - B(\cc)^{2N}e_j= B(\cc)^{2N}[i]-B(\cc)^{2N}[j],
\end{align*}
where $e_i,e_j$ are vectors of the canonical basis of $\R^{2N}$ and $B(\cc)^{2N}[i],B(\cc)^{2N}[j]$ denote the $i$-th and 
$j$-th column of the matrix $B(\cc)^{2N}$. 
Hence, $\|B(\cc)^{2N}\vvv_{ij}\|_{\ell_1}$ corresponds to the distance between two columns of $B(\cc)^{2N}$ 
indicized by either $i,j\in\mathcal{I}'$  or  $\in\mathcal{I}''$.

Assume that $i,j\in\mathcal{I}'$, the other case being completely similar. 
We use the expression \eqref{eq:expansion} for $B(\cc)^{2N}$ and Theorem~\ref{theo:main} to get
\begin{align*}
\left\|B(\cc)^{2N}[i]-B(\cc)^{2N}[j]\right\|_{\ell_1}=\left(1+ \frac {d}N \right)^{-2N} \sum_{\ell=1}^{2N}|b_{\ell i}-b_{\ell j}|,
\end{align*}
where $b_{\ell i}$ denotes the generic element of the matrix $[B(0) + \gamma B_1]^{2N}$ and
where $b_{\ell i}$, $b_{\ell j}=0$ if $\ell\notin\mathcal{I}'$. 

A key observation is that, by applying formula \eqref{eq:theoB2N-bis} and recalling the definition \eqref{def:hat_P} of $\widehat P$, 
the contribution from the term $2{\frac dN} \widehat P$ is zero because
$$
\widehat P[i] - \widehat P[j] = 0\in \R^{2N} \,,\qquad i,j\in\mathcal{I}'\,.
$$
The same property holds if $i,j\in\mathcal{I}''$. Therefore
\begin{align*}
\sum_{\ell=1}^{2N}|b_{\ell i}-b_{\ell j}|&\le |b_{i i}-b_{i j}| + |b_{j i}-b_{j j}| + \sum_{\ell\not = i, j}^{2N}|b_{\ell i}-b_{\ell j}|\\
&\le 2 \left(1 + \sum_{j=0}^{2N-1} \zeta_{j,N} +  \sum_{j=1}^{2N-1} \eta_{j,N} \right)   \\ 
&\le 2 \left( \sinh(2d)  - 2d+ {\frac1N} f_0(d)  +  \cosh(2d)  ~+~  {\frac1N} f_1(d)  \right)  \\
&=\left\|\vvv_{ij}\right\|_{\ell_1}\left[e^{2d}-2d   ~+~  {\frac1N} [f_0(d)  + f_1(d)]  \right].
\end{align*}
By denoting 
\begin{equation*}
C_N(d) \dot = \left(1+ \frac {d}N \right)^{-2N} \left[e^{2d}-2d   ~+~  {\frac1N} [f_0(d)  + f_1(d)]  \right]\,,
\end{equation*}
we easily get that
$
C_N(d) \to (1 - 2d\ee^{-2d})$ as  $N\to\infty$\,, and this completes the proof of Proposition~\ref{prop:5.8}\,.
\end{proof}

\subsection{Nonlinear damping}\label{subsec:5.4}
In this subsection we prove Theorem~\ref{main-theorem}.  

Assume that \eqref{hyp:k-unif-positive} holds, that is $0<k_1\leq k(x)\leq k_{2}$ for some positive $k_1, k_2$ and recall the definition of 
$0<d_1\le d_2$ given in \eqref{d_*}. We study the behavior of 
$$
\BB_{2N}  = \left[ B^{(2N)} B^{(2N-1)}\cdots  B^{(2)}  B^{(1)}\right] \,.
$$ 
By the inequality \eqref{inequality-B} we have
\begin{equation*}
B(\cc^{n})\leq \left(1+ \frac{d_1}N \right)^{-1} \left[B(0) + \frac{d_2}N B_{1} \right]\qquad \forall\, n\,,
\end{equation*}
and then
\begin{equation}\label{BB_2N}
\BB_{2N}\le  \left(1+ \frac{d_1}N \right)^{-2N} \left[B(0) + \frac{d_2}N B_{1} \right]^{2N}\,.
\end{equation}
\begin{proposition}\label{nonlinear-damping}
	There exists a constant $C_N(d_1,d_2)$ such that  as $N\to\infty$
	\begin{equation}\label{def:C(d_1,d_2)}
	C_N(d_1,d_2) \to \ee^{-2d_1}(\ee^{2d_2}-2d_2)~\dot{=}~C(d_1,d_2)
	\end{equation}
	and that for $i,j$ indices fixed either $\in\mathcal{I}'$ or $\in\mathcal{I}''$ it holds
	\begin{equation*}
	\bigl\|\BB_{2N}\vvv_{ij} \bigr\|_{\ell_1} \le C_N(d_1,d_2) \bigl\|\vvv_{ij}\bigr\|_{\ell_1}\,.
	\end{equation*}
	In particular, if $d_1$ and $d_2$ satisfy \eqref{cond-on-d_*}, then $C_N(d_1,d_2)<1$ for $N$ large enough.
\end{proposition}

\begin{proof} From \eqref{BB_2N}, one can estimate the term $ \left[B(0) + \frac{d_2}N B_{1} \right]^{2N}$ on the right hand side
as in the proof of Theorem~\ref{theo:main}. Then as in the proof of Proposition~\ref{prop:5.8}, the conclusion follows easily with 
\begin{equation*}
C_N(d_1,d_2) \dot = \left(1+ \frac {d_1}N \right)^{-2N} \left[\ee^{2d_2}-2d_2   ~+~  {\frac1N} [f_0(d_2)  + f_1(d_2)]  \right]\,.
\end{equation*}
\end{proof}

\begin{proof}[of Theorem 1.1] 
To prove \eqref{L-infty-decay} in Theorem~\ref{main-theorem} we employ the main results in this section, namely
Proposition~\ref{prop:5.2}, Lemma~\ref{lem:decomp}, Theorem~\ref{theo:main} and Proposition~\ref{nonlinear-damping}. 
About the estimate for $J$, we proceed as follows. 

\begin{itemize}

\smallskip
\item[$\bullet$] We start from \eqref{bound-on-sup-J}, that is
\begin{equation*}
\|J_\DX(\cdot,t)\|_{\infty} \le \frac{1}{2N} \tv \bar J_0 
+ \|\BB_n  \widetilde{\ssigma}(0+)\|_{\ell^1} \,.
\end{equation*}

\smallskip
\item[$\bullet$]  Let $n\in \N$, $0\le h\in \N $ and $2Nh \le n < 2N(h+1)$, so that
\begin{equation}\label{rel:n-h-N}
2 h \le \frac{n}{N} = {n\DT} = {t^n} < 2(h+1)\,,\qquad h\ge 0\,.
\end{equation}
Since ${E_-}$ is an invariant subspace for all $B^{(n)}$, 
we have
$$
\widetilde{\ssigma}(t^n)=\BB_n\widetilde{\ssigma}(0+) \in {E_-} \quad \forall \, n\,.
$$
Hence by Proposition~\ref{nonlinear-damping} and using that 
$\bigl\|B^{(n)} v \bigr\|_{\ell_1}\le \bigl\| v \bigr\|_{\ell_1}$ for all $v\in \R^{2N}$, the following holds
\begin{align*}
\bigl\|\widetilde{\ssigma}(t^{n})\bigr\|_{\ell_1} =\bigl\|\BB_n\widetilde{\ssigma}(0+)\bigr\|_{\ell_1} \le & ~\bigl\|\BB_{2Nh} \widetilde{\ssigma}(0+)\bigr\|_{\ell_1}\\
=  & ~ \bigl\|\BB_{2N}\left(\BB_{2N(h-1)} \widetilde{\ssigma}(0+)\right)\bigr\|_{\ell_1}\\
\le & ~ C_{N}  \bigl\| \BB_{2N(h-1)} \widetilde{\ssigma}(0+) \bigr\|_{\ell_1}\\
\le & ~ C_{N}^{h}  \bigl\| \widetilde{\ssigma}(0+)\bigr\|_{\ell_1}\,.
\end{align*}
\smallskip
Let $\delta>0$ satisfy $[C-\delta,C+\delta]\subset (0,1)$, and choose 
$N$ large enough so that $C_N(d_1,d_2)\in [C-\delta,C+\delta]$. One can easily get
\begin{align*}
\left|C_N (d_1,d_2)- C(d_1,d_2)\right| \le ~ & \frac{1}{N} \left(1+ \frac {d_1}N \right)^{-2N}\left[f_0(d_2)  + f_1(d_2)\right] \\
& \qquad + \left(\ee^{2d_2}-2d_2\right)\cdot \ee^{-2d_1}\left(\left(1+ \frac {d_1}N \right)^{2}~-1\right)\\
\le ~& \frac{1}{N} \hat{C}(d_1,d_2)
\end{align*}
for a suitable constant $\hat{C}(d_1,d_2)>0$. Therefore one has
$$
\left|C_{N}^{h}-C^{h}\right|\le \left|C_{N}-C\right|\cdot h |\xi|^{h-1}, \qquad \forall \ h\ge 1\,,
$$
for some $\xi \in [C-\delta,C+\delta] \subset (0,1)$. Since the quantity $h |\xi|^{h-1}$ is uniformly bounded for $h\ge 1$ and 
$\xi\in [C-\delta, C+\delta]$, then we deduce that for some $ \hat{C}_0>0$ one has
\begin{equation*}
\bigl\|\BB_n\widetilde{\ssigma}(0+)\bigr\|_{\ell_1} \le\left(C^h + \frac{\hat{C}_0}N \right)  \bigl\| \widetilde{\ssigma}(0+)\bigr\|_{\ell_1}
\end{equation*}
where $n$, $N$, $h$ satisfy \eqref{rel:n-h-N}.

\smallskip
\item[$\bullet$]  From \eqref{eq:decomp-0} we have that
$$
 \widetilde{\ssigma}(0+)=\ssigma (0+) - \frac{(\ssigma(0+)\cdot v_-)}{2N}\,v_-,
$$
and then
\begin{equation*}
\bigl\| \widetilde{\ssigma}(0+)\bigr\|_{\ell_1}\le \bigl\| \ssigma(0+)\bigr\|_{\ell_1}+\frac{\bigl\|\ssigma(0+)\bigr\|_{\ell_1}}{2N} 2N 
= 2\bigl\| \ssigma(0+)\bigr\|_{\ell_1}.
\end{equation*}
Moreover, using \eqref{ineq:sizes} and \eqref{size-at-boundary}, we have
\begin{equation*}
\bigl\| \ssigma(0+)\bigr\|_{\ell_1} \le\, \tv \rho_{0}  + \tv \bar{J}_{0} +  2 C_0 \|k\|_{L^1}
\end{equation*}
where $\bar J_0$ is defined at \eqref{def:bar_J0}. Therefore it holds, for $h \le\frac{t^n}{2}\le (h+1)$:
\begin{equation*}
\bigl\|\BB_n\widetilde{\ssigma}(0+)\bigr\|_{\ell_1} \le 2 \left(C^h + \frac{\hat{C}_0}N \right) \left( \tv \rho_{0}  + \tv \bar{J}_{0} +  2 C_0 \|k\|_{L^1}\right)\,.
\end{equation*}
Using the relation \eqref{rel:n-h-N} for $h$, $n$ and $N$, we have
\begin{align*}
C^{h}\le ~ C^{\frac {t^{n}}{2} -1}= & ~ \frac 1 C \ee^{-|\log C|\left(\frac {t^{n}}{2}\right)}\,.
\end{align*}
In conclusion we get
\begin{align*}
\|J_\DX(\cdot,t^n)\|_{\infty} & \le \frac{1}{2N} \left\{\tv \bar J_0 +  4 {\hat{C}_0}  \left( \tv \rho_{0}  + \tv \bar{J}_{0} +  2 C_0 \|k\|_{L^1}\right)\right\}\\
& \qquad +  \frac 2 C \ee^{-|\log C|\left(\frac {t^{n}}{2}\right)} \left( \tv \rho_{0}  + \tv \bar{J}_{0} +  2 C_0 \|k\|_{L^1}\right)
\end{align*}
that leads to the first inequality in \eqref{L-infty-decay} for suitable constants $\hat C_j$ which are independent of $\DX$ and $t$. 
The constant $\hat C_3$ is given by
\begin{equation*}
\hat C_3 = \frac{1}{2} |\log C(d_1,d_2)| \qquad     C(d_1,d_2) = \ee^{-2d_1}(\ee^{2d_2}-2d_2)\,.
\end{equation*}
\end{itemize}
Starting from \eqref{bound-on-sup-rho}, the second inequality in \eqref{L-infty-decay}, for the $\rho$ variable, is obtained in a similar way.
\end{proof}
\thanks{\textsl{Acknowledgements.} 
This research was partially supported by Miur-PRIN 2015, 
Grant No. 2015YCJY3A\_003, and by 2018 INdAM-GNAMPA Project "Equazioni iperboliche e applicazioni". 
The authors wish to thank Pietro Dell'Acqua, Laurent Gosse, Nicola Guglielmi, Fabrizio Nesti, 
Michael L.\ Overton, Philippe Thieullen and Enrique Zuazua for several interesting discussions done at various stages of this work. 
Finally we thank both referees for their careful reading and their valuable comments.}

\affiliationone{%
   D. Amadori and F. Aqel and E. Dal Santo \\
   DISIM, University of L'Aquila\\
   Italy
   \email{debora.amadori@univaq.it\\
              fatimaalzahraan.aqel@graduate.univaq.it\\ 
              dalsantoedda@gmail.com}}
\end{document}